\documentclass[review,onefignum,onetabnum]{siamart171218}

\newcommand{\Qvec}{\mathbf{Q}}

\newcommand{\nvec}{\mathbf{n}}

\newcommand{\eps}{\epsilon}
\usepackage{tikz}
\usepackage{bm}
\usepackage{mathtools}
\usepackage{nicematrix}
\graphicspath{{Pictures/}}

\usepackage{lipsum}
\usepackage{amsfonts}
\usepackage{graphicx}
\usepackage{epstopdf}
\usepackage{algorithmic}
\ifpdf
  \DeclareGraphicsExtensions{.eps,.pdf,.png,.jpg}
\else
  \DeclareGraphicsExtensions{.eps}
\fi

\newsiamremark{remark}{Remark}
\newsiamremark{hypothesis}{Hypothesis}
\crefname{hypothesis}{Hypothesis}{Hypotheses}
\newsiamthm{claim}{Claim}

\headers{Order Reconstruction in Microfluidic Channels}{J. Dalby, Y. Han, A. Majumdar, L. Mrad}

\title{A Multi-Faceted Study of Nematic Order Reconstruction in Microfluidic Channels.\thanks{Submitted to the editors DATE.
\funding{AM is supported by the University of Strathclyde New Professors Fund, a Leverhulme International Academic Fellowship, an OCIAM Visiting Fellowship at the University of Oxford and a Daiwa Foundation Small Grant. JD acknowledges support from the University of Strathclyde and the DST-UKIERI. YH is supported by a Royal Society Newton International Fellowship.}}}

\author{James Dalby\thanks{Department of Mathematics, University of Strathclyde, UK
  (\email{james.dalby@strath.ac.uk}, \email{apala.majumdar@strath.ac.uk},
  \email{yucen.han@strath.ac.uk}).}
  \and Yucen Han\footnotemark[2]
  \and Apala Majumdar*\footnotemark[2]
\and Lidia Mrad\footnotemark[3]\thanks{Department of Mathematics and Statistics, Mount Holyoke College, Massachusetts, USA 
  (\email{lmrad@mtholyoke.edu}).}
}

\usepackage{amsopn}

\begin{document}
\nolinenumbers
\maketitle

\begin{abstract}
  We study order reconstruction (OR) solutions in the Beris-Edwards framework for nematodynamics, for both passive and active nematic flows in a microfluidic channel. OR solutions exhibit polydomains and domain walls, and as such, are of physical interest. We show that OR solutions exist for passive flows with constant velocity and pressure, but only for specific boundary conditions. We prove the existence of unique, symmetric and non-singular nematic profiles, for boundary conditions that do not allow for OR solutions. We compute asymptotic expansions for OR-type solutions for passive flows with non-constant velocity and pressure, and active flows, which shed light on the internal structure of domain walls. The asymptotics are complemented by numerical studies that demonstrate the universality of OR-type structures in static and dynamic scenarios. 
\end{abstract}

\begin{keywords}
  Nematodynamics, Active liquid crystals, Microfluidics 
\end{keywords}

\begin{AMS}
  34B60, 34E10, 76A15
\end{AMS}

\section{Introduction}\label{sec:intro}
Nematic liquid crystals (NLCs) are mesophases that combine fluidity with the directionality of solids \cite{deGennes}. The NLC molecules tend to align along certain locally preferred directions, leading to a degree of long-range orientational order. The orientational ordering results in direction-dependent physical properties that render them suitable for a range of industrial applications, including optical displays. When confined to thin planar cells and in the presence of fluid flow, applications of nematics are further extended, for example, to optofluidic devices and guided micro-cargo transport through microfluidic networks \cite{optofluidic, microcargo}. These hydrodynamic applications are facilitated by the coupling between the fluidity and the orientational ordering, leading to exceptional mechanical and rheological properties \cite{mondal}. 

Flow-induced deformation of nematic textures in confinement are ubiquitous, both in passive systems where the hydrodynamics are driven by external agents, as well as in active systems. Active matter systems, composed of self-driven units, also exhibit orientational ordering and collective motion, resulting in a wealth of intriguing non-equilibrium properties \cite{marchetti}. We focus on passive and active nematodynamics in microfluidic channels, with a view to model spatio-temporal pattern formation and to analyse the stability of singular lines or domain walls in such channels. 

We work with long, shallow, three-dimensional (3D) microfluidic channels of width $L$, in a reduced Beris-Edwards framework \cite{berisreference}. Our domain is effectively one-dimensional (1D), since we assume that structural details are invariant across the length and height of the channel. We work with a reduced Landau-de Gennes (LdG) $\Qvec$-tensor for the nematic ordering. This reduced $\Qvec$-tensor has two degrees of freedom 
- the planar nematic director, $\mathbf{n}$, in the two-dimensional (2D) channel cross-section, and an order parameter, $s$, related to the degree of nematic ordering. The director $\mathbf{n}$ is parameterised by an angle, $\theta$, which describes the in-plane alignment of the nematic molecules. In a fully 3D framework, the LdG $\Qvec$-tensor has five degrees of freedom and there are exact connections between the reduced LdG and the 3D LdG descriptions, as discussed in the next section. 
We consider steady unidirectional flows, which, within the Beris-Edwards framework, are captured by a system of coupled differential equations for $s$, $\theta$, and the fluid velocity $\mathbf{u}$. There are three dimensionless parameters, 
the most important of which is $L^*$, which is inversely proportional to $L^2$ and plays a key role in the stability of singular structures.

Our work is largely devoted to {\it Order Reconstruction} (OR) solutions (defined precisely in section \ref{sec:constant-pressure-flow}). OR solutions are nematic profiles with distinct director polydomains, separated by singular lines or singular surfaces, referred to as domain walls. Mathematically, the domain walls are 
simply disordered regions in the plane, and would appear as singularities in 2D optical studies but in 3D, they describe a continuous yet rapid rotation between distinct 3D NLC configurations in the two adjacent polydomains, as in the seminal paper \cite{sluckin}. OR solutions are relevant for modelling chevron or zigzag patterns observed in pressure-driven flows \cite{agha, copar}, as well as in active nematics where aligned fibers can be controlled to display a laminar flow \cite{guillamat}. 
 OR solutions have been studied in purely nematic systems, for example in \cite{lamy}, \cite{canevari} and \cite{Harris}. However, they are not limited to purely nematic systems: for instance, OR solutions exist in ferronematic systems comprising magnetic nanoparticles in NLC media \cite{dalby-farrell-majumdar-xia}. Generalized OR solutions or OR-type solutions/instabilities (defined in section \ref{sec:passive+active}) are also observed in smectics. For example, when a cell filled with a smectic-A liquid crystal is cooled to the smectic-C phase, a chevron texture is observed and has been the impetus of considerable experimental and theoretical interest \cite{RiekerChevron, MottramBiaxialChevron, PhillipsChevron}.

We thus speculate that OR solutions are a universal property of partially ordered systems, especially small systems with conflicting boundary conditions. For systems with constant velocity and constant pressure, 
we prove that OR solutions only exist for mutually orthogonal boundary conditions imposed on $\theta$. It is known that OR solutions are compatible with orthogonal boundary conditions and we prove that this is the only possibility. 
For all other choices of Dirichlet boundary conditions for $\theta$, we show that OR solutions do not exist and using geometric and comparison principles, we prove the existence of a unique, symmetric and non-singular $(s, \theta)$-profile in these cases. 
 For general flows with non-constant velocity and pressure, in section \ref{sec:passive+active}, we work with large domains ($L^*\to0$) and compute asymptotic approximations for \emph{OR-type solutions}, that exhibit a singular line or domain wall in the channel centre, for both passive and active scenarios. For OR-type solutions, the director is not constant away from the isotropic line, as in the case of OR solutions. 
 Our asymptotic methods are adapted from \cite{CM}, where the authors investigate a chevron texture characterised specifically by a $\pm \pi/4$ jump in $\theta$, using an Ericksen model for uniaxial NLCs. 
 These asymptotic methods, now placed within the Beris-Edwards framework, allow us to explicitly construct OR-type solutions with a planar jump discontinuity in $\theta$. 
 We also construct OR-type solutions for active nematodynamics, by working in the reduced Beris-Edwards framework with additional non-equilibrium active stresses \cite{active-defects}, thus illustrating the universality of OR-type solutions.

 We validate our asymptotics 
 for passive and active nematodynamics (with non-constant pressure and flow), with extensive numerical experiments, for large and small values of $L^*$. In both settings, we find OR-type solutions for all values of $L^*$, with mutually orthogonal Dirichlet conditions for $\theta$ on the channel walls. OR-type solutions are stable for large $L^*$, and unstable for small $L^*$. In fact, we observe multiple unstable OR-type solutions for small values of $L^*$. 
 Our asymptotic expansions serve as excellent initial conditions for numerically computing different branches of OR-type solutions, characterised by different jumps in $\theta$
 , and the numerics agree well with the asymptotics. We speculate that unstable OR-type solutions can potentially be stabilised by external controls and thus, play a role in switching and dynamical phenomena.

The paper is organised as follows. In section \ref{sec:theory}, we describe the Beris-Edwards model, our channel geometry and the imposed boundary conditions. In section \ref{sec:constant-pressure-flow}, we study flows with constant velocity and pressure, and identify conditions which allow and disallow OR solutions, in terms of the boundary conditions. 
In section \ref{sec:passive+active}, we compute asymptotic expansions for OR-type solutions with passive and active nematic flows for small $L^*$ or large channel widths, providing explicit limiting profiles in these cases. We then supplement our analysis with detailed numerical experiments, followed by
some brief conclusions and future perspectives in section \ref{sec:conclusions}.

\section{Theory}\label{sec:theory}
We consider NLCs sandwiched inside a three-dimensional (3D) channel, $\tilde{\Omega}=\{(x,y,z)\in\mathbb{R}^3:-D\leq x\leq D,-L\leq y\leq L,0\leq z\leq H\}$ where $L, D,$ and $H$ are the (half) width, length and full height of the channel, respectively. We assume that $D\gg L$ and $H \ll L$. We further assume planar surface anchoring conditions on the top and bottom channel surfaces at $z=0$ and $z=H$, which effectively means that the NLC molecules lie in the $xy$-plane on these surfaces, without a specified direction. Such boundary conditions are used in experiments, see for example the planar bistable nematic device in \cite{tsakonas} and the experiments on fd-viruses in \cite{lewis2014}. We impose $z$-invariant Dirichlet conditions on $y=\pm L$ and periodic conditions on $x=\pm D$, compatible with the planar conditions on $z =0, H$. Given the planar surface anchoring conditions on the top and bottom surfaces and that the well height is small, we assume that the system is invariant in the $z$-direction. Furthermore, 
since $D\gg L$, we assume that the system is invariant in the $x$-direction and this reduces our computational domain to a 1D channel, $y\in[-L,L]$.  

In the LdG framework, the $\Qvec$-tensor order parameter is a symmetric, traceless $3\times 3$ matrix, with five degrees of freedom. Given the modelling assumptions above regarding invariance in the $z$-direction, we assume that the physically relevant NLC ($\Qvec$) configurations belong to a reduced space of $\Qvec$-tensors that have a fixed eigenvector in the $z$-direction and an associated constant eigenvalue. This reduces the degrees of freedom from five to simply two degrees of freedom, as captured by the reduced LdG $\Qvec$-tensor in \eqref{eq:Q} below. Under these assumptions, the full LdG $\Qvec$-tensor 
is simply the sum of the reduced $\Qvec$-tensor and a constant $3\times 3$ matrix. See the supplementary material for an explicit example connecting the reduced and full LdG $\Qvec$-tensors. The reduced approach can be rigorously justified, in some cases, by gamma convergence methods; see Theorem 5.1 in \cite{golovaty2015} (and Theorem 2.1 in \cite{wang_canevari_majumdar_2019}) where the authors show that for planar surface anchoring conditions on $z=0, H$, and for Dirichlet conditions on the lateral surfaces, the minimizers of the LdG energy do indeed have a fixed eigenvector in the $z$-direction with constant eigenvalue, in the $H\to 0$ limit, and the reduced $\Qvec$-tensor suffices for modelling purposes. We do not give rigorous proofs in this paper, given that our work is in the spirit of formal mathematical modelling. 

There are two macroscopic variables in our reduced framework: the fluid velocity $\mathbf{u}$, and a reduced LdG $\mathbf{Q}$-tensor order parameter that measures the NLC orientational ordering in the $xy$-plane. 
More precisely, the reduced $\Qvec$-tensor is a symmetric traceless $2\times 2$ matrix i.e., $\Qvec\in S_2\coloneqq\{\Qvec\in\mathbb{M}^{2\times 2}: Q_{ij}=Q_{ji},Q_{ii}=0\}$, which can be written as:
\begin{equation}
    \label{eq:Q}
    \Qvec = s \left(\mathbf{n}\otimes \mathbf{n} - \frac{\mathbf{I}}{2}\right).
\end{equation}
Here, $s$ is a scalar order parameter, $\mathbf{n}$ is the nematic director (a unit vector describing the average direction of orientational ordering in the $xy$-plane), and $\mathbf{I}$ is the $2\times 2$ identity matrix.
Moreover, $s$ can be interpreted as a measure of the degree of order about $\mathbf{n}$, so that the nodal sets of $s$ (i.e., where $s=0$) define nematic defects in the $xy$-plane. 
As a consequence of \eqref{eq:Q}, 
the two independent components of $\Qvec$ are given by
\begin{equation}
            Q_{11} = \frac{s}{2}\cos 2 \theta,\quad Q_{12} = \frac{s}{2} \sin 2\theta,\label{eq:Q-components}
\end{equation}
when $\mathbf{n} = \left(\cos \theta, \sin \theta \right)$, and $\theta$ is the angle between $\mathbf{n}$ and the $x$-axis. Conversely, applying basic trigonometric identities, we have the following relationships,
\begin{equation}
    s=2\sqrt{Q^2_{11}+Q^2_{12}}\quad\textrm{and}\quad \theta=\frac{1}{2}\tan^{-1}\left(\frac{Q_{12}}{Q_{11}}\right). \label{eq:s-theta-relationship}
\end{equation}

We work within the Beris-Edwards framework for nematodynamics \cite{berisreference}. There are three governing equations: an incompressibility constraint for $\mathbf{u}$, an evolution
equation for $\mathbf{u}$ (essentially the Navier–Stokes equation with an additional stress due to the nematic ordering, $\sigma$), and an evolution equation for $\Qvec$ which has an additional stress induced by the fluid vorticity \cite{mondal}. These equations are given below,
\begin{eqnarray*}
    && \nabla\cdot\mathbf{u}=0,\quad\rho\frac{D\mathbf{u}}{Dt}=-\nabla p +\nabla\cdot(\mu(\nabla \mathbf{u}+(\nabla\mathbf{u})^T)+\sigma),\\
    && 
    \frac{D\Qvec}{Dt}=\mathbf{\zeta}\Qvec-\Qvec\mathbf{\zeta}+\frac{1}{\gamma}\mathbf{H}.
\end{eqnarray*}
Here $\rho$ and $\mu$ are the fluid density and viscosity respectively, $p$ is the hydrodynamic pressure, $\mathbf{\zeta}$ is the anti-symmetric part of the velocity gradient tensor and $\gamma$ is the rotational diffusion constant. The nematic stress is defined to be
\begin{equation*}
   \sigma = \Qvec\mathbf{H}-\mathbf{H}\Qvec\quad\textrm{and}\quad \mathbf{H}=\kappa\nabla^2\Qvec-A\Qvec-C|\Qvec|^2\Qvec,
\end{equation*}
where $\mathbf{H}$ is the molecular field related to the LdG free energy, $\kappa$ is the nematic elasticity constant, $A<0$ is a temperature dependent constant, $C>0$ is a material dependent constant, and $| \Qvec |=\sqrt{\textrm{Tr}(\Qvec^T\Qvec)}$, is the Frobenius norm. Finally, 
we assume that all quantities depend on $y$ alone and work with a unidirectional channel flow, so that $\mathbf{u}=(u(y),0)$. The incompressibility constraint is automatically satisfied. To render the equations nondimensional, we use the following scalings, as in \cite{mondal},
\[y = L\tilde y,\; t = \frac{\gamma L^2}{\kappa}\tilde t,\; u = \frac{\kappa}{\gamma L}\tilde u,\; Q_{11} = \sqrt{\frac{-2A}{C}}\tilde Q_{11}, \; Q_{12} = \sqrt{\frac{-2A}{C}}\tilde Q_{12}, \; p_x = \frac{\mu \kappa}{\gamma L^3}\tilde p_x,\]
and then drop the tilde for simplicity. Our rescaled domain is $\Omega=[-1,1]$ and the evolution equations become
\begin{subequations}\label{eq:Q-flow}
    \begin{align}
        & \frac{\partial Q_{11}}{\partial t}=u_y Q_{12}+ Q_{11,yy}+\frac{1}{L^*}Q_{11}(1-4(Q_{11}^2+Q_{12}^2)), \label{eq:Q-flow-a}\\
        & \frac{\partial Q_{12}}{\partial t}=-u_y Q_{11}+ Q_{12,yy}+\frac{1}{L^*}Q_{12}(1-4(Q_{11}^2+Q_{12}^2)), \label{eq:Q-flow-b}\\
        & L_1\frac{\partial u}{\partial t}=-p_x+u_{yy}+2L_2(Q_{11}Q_{12,yy}-Q_{12}Q_{11,yy})_y,\label{eq:Q-flow-c}
    \end{align}
\end{subequations}
where $L_1=\frac{\rho\kappa}{\mu\gamma}$, $L^*=\frac{-\kappa}{AL^2}$, and $L_2=\frac{-2A\gamma}{C\mu}=\frac{-2AEr^*}{CEr}$ are dimensionless parameters. Here, $Er=u_0L\mu/\kappa$ is the Ericksen number and $Er^*=u_0L\gamma /\kappa$ ($u_0$ is the characteristic length scale of the fluid velocity) is analogous to the Ericksen number in terms of the rotational diffusion constant $\gamma$, rather than viscosity $\mu$. We interpret $L^*$ as a measure of the domain size i.e. it is the square of the ratio of two length scales: the nematic correlation length, $\xi = \sqrt{-\kappa/ A}$ for $A<0$ and the domain size $L$, so that the $L^* \to 0$ limit is relevant for large channels or macroscopic domains. The parameter, $L_2$ is the product of the ratio of material and temperature-dependent constants and the ratio of rotational to momentum diffusion \cite{mondal}. In what follows, we fix $L_1=1$, and as such do not comment on its physical significance.  The static governing equations for $(s,\theta)$, can be obtained from (\ref{eq:Q-flow}) using (\ref{eq:Q-components}):
\begin{subequations}
\label{eq:flow-system}
\begin{align}
& s_{yy}=4s\theta^2_y+\frac{1}{L^*}s(s^2-1),\label{eq:s-eqtn}\\
& s\theta_{yy}=\frac{1}{2}s u_y-2s_y\theta_y,\label{eq:theta-eqtn}\\
& u_{yy}=p_x-L_2(s^2\theta_y)_{yy}\label{eq:u-eqtn}.
\end{align}
\end{subequations}
The formulation in terms of $(s, \theta)$ gives informative insight into the solution profiles 
and avoids some of the degeneracy conditions coded in the $\Qvec$-formulation.

We work with Dirichlet conditions for $(s,\theta)$ 
as given below:
\begin{subequations}
\label{eq:BCs}
\begin{align}
& s(-1)=s(1)=1,\label{eq:s-bc}\\
& \theta(-1)=-\omega\pi,\;\theta(1)=\omega\pi,\label{eq:theta-bc}
\end{align}
\end{subequations}
where
   $ \omega
    \in\left[-\frac{1}{2},\frac{1}{2}\right]$,
 is the winding number. This translates to the following boundary conditions for $\Qvec$:
\begin{equation}
    Q_{11}(\pm 1)=\frac{1}{2}\cos(2\omega\pi),\;Q_{12}(-1)=-\frac{1}{2}\sin(2\omega\pi),\;Q_{12}(1)=\frac{1}{2}\sin(2\omega\pi).\label{eq:Q-bcs}
\end{equation}
The boundary conditions in \eqref{eq:s-bc} imply that the nematic molecules are perfectly ordered on the bounding plates. We consider asymmetric Dirichlet boundary conditions in \eqref{eq:theta-bc} for the angle $\theta$. A potential issue follows from \eqref{eq:s-theta-relationship}: the range of $\theta$ is $(-\frac{\pi}{4},\frac{\pi}{4})$, but our boundary conditions extend to $\pm\frac{\pi}{2}$. However, we circumvent this issue by using the function atan2$(y,x)\in(-\pi,\pi]$, which returns the angle between the line connecting the point $(x,y)$ to the origin and the positive $x$ axis.
\begin{figure}[ht]
    \centering
    \begin{minipage}{1.0\textwidth}
        \centering
        \includegraphics[width=0.8\textwidth]{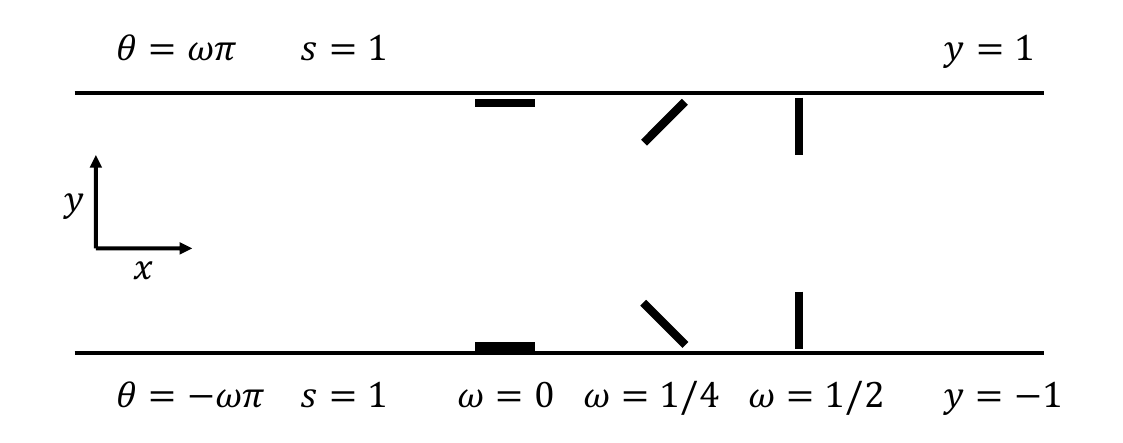}\\
    \end{minipage}
\caption{Boundary conditions for $s$ and $\theta$, and some example boundary conditions on the director.}
    \label{fig:BCs-s-theta}
\end{figure}
\noindent
For the flow field, we consider the typical no-slip boundary conditions, namely 
\begin{align}
    u(-1)=u(1)=0,\label{eq:poiseuille}
\end{align}
and assume that the pressure $p$ is uniform in the $y$-direction, depending on $x$ only.


\section{Passive flows with constant velocity and pressure}\label{sec:constant-pressure-flow}

In this section, we study nematic flows with constant velocity and pressure without additional activity. This framework, though somewhat artificial, allows for OR solutions, although OR-type solutions exist in more generic situations with non-constant flows. We work with both the $\Qvec$- and $(s,\theta)$-frameworks in this section.

In 
our one-dimensional framework, OR solutions correspond to a partition of the domain $\Omega=[-1,1]$ into sub-domains, $\Omega = \sum_{j=1}^{n}\Omega_j$, where each $\Omega_j$ is a \emph{polydomain}. These polydomains have
constant $\theta$ (recall that $\theta$ is the orientation of the planar director, $\mathbf{n}$), separated by domain walls (with $s=0$) to account for planar jumps in $\theta$ across polydomain boundaries. 
OR-type solutions are simply interpreted as solutions of (\ref{eq:Q-flow}) that have a non-empty nodal set for $s$ or exhibit domain walls, without the constraint of constant $\theta$ in each polydomain. 
In the reduced $\Qvec$-framework, OR solutions have distinct but less obvious signatures and the domain walls correspond to the nodal set of the reduced $\Qvec$-tensor. 
In a 3D LdG $\Qvec$-description, the corresponding nematic director rapidly rotates between two distinct director profiles across the domain wall, and the rotation is mediated by maximal biaxiality; see supplementary material.
We show, below, that OR-solutions are only compatible with specific boundary conditions in the $\Qvec$-framework.

In the $(s, \theta)$-framework, OR solutions are characterised by sub-intervals with constant $\theta$. 
From \eqref{eq:theta-eqtn}, constant $\theta$  implies constant fluid velocity $u$ and from \eqref{eq:u-eqtn}, constant pressure, $p$. Therefore, we assume constant velocity and pressure to start with. In what follows, $\prime$ denotes differentiation with respect to $y$. 

In this scenario the static version of (\ref{eq:Q-flow-a})-(\ref{eq:Q-flow-b}) is 
\begin{subequations}\label{eq:Q-EL-eqnts}
\begin{align}
    & Q^{\prime\prime}_{11}=\frac{1}{L^*} Q_{11}(4(Q^2_{11}+Q^2_{12})-1), \label{eq:Q11eqtn}\\
    & Q^{\prime\prime}_{12}=\frac{1}{L^*} Q_{12}(4(Q^2_{11}+Q^2_{12})-1).\label{eq:Q12eqtn}
\end{align}
\end{subequations}

\noindent
From these equations it follows that \eqref{eq:Q-flow-c} is satisfied. The equations (\ref{eq:Q11eqtn})-(\ref{eq:Q12eqtn}) are the Euler-Lagrange equations associated with the energy
\begin{align}
    F_{LG}[Q_{11}, Q_{12}]=
    \int_{\Omega}\left(\left(Q_{11}^\prime\right)^2+\left(Q_{12}^\prime\right)^2\right)+\frac{1}{L^*}(Q^2_{11}+Q^2_{12})(2(Q^2_{11}+Q^2_{12})-1 )~\mathrm{d}y. \label{eq:q-energy}
\end{align}
The admissible $\Qvec$-tensors belong to the Sobolev space, $W^{1,2}\left([-1,1]; S_2 \right)$, where $S_2$ is the space of symmetric and traceless $2\times 2$ matrices, subject to appropriately defined boundary conditions (see (\ref{eq:Q-bcs})). The stable and physically observable configurations correspond to local or global minimizers of (\ref{eq:q-energy}), in the prescribed admissible space.

In the static case, with constant $u$ and $p$, the corresponding equations for $(s, \theta)$ can be deduced from \eqref{eq:s-eqtn}, \eqref{eq:theta-eqtn} :
\begin{subequations}
\label{eq:EL-eqtns}
\begin{align}
& s^{\prime\prime}=4s(\theta^\prime)^2+\frac{1}{L^*}s(s^2-1),\label{eq:s-no-flow}\\
& 
\label{eq:theta-no-flow}
\left(s^2\theta^\prime\right)^\prime=0,\implies s^2\theta^\prime=B,
\end{align}
\end{subequations}
whilst \eqref{eq:u-eqtn} is automatically satisfied. In the above, $B$ is a fixed constant of integration; in fact
\begin{equation}
    B=\theta^\prime(-1)=\theta^\prime(1).\label{eq:B}
\end{equation}

When $\omega\geq0$ and recalling the boundary conditions for $\theta$, there exists a point $y_0$ such that $\theta^\prime(y_0)\geq0$, hence $B\geq0$, and $\theta^\prime\geq 0$ for all $y\in[-1,1]$.
Thus, we have
\begin{equation}\label{eq:theta-bounds}
    -\omega\pi\leq\theta\leq\omega\pi,\;\forall y\in[-1,1]\textrm{ and }\forall\omega\in\left[0,\frac{1}{2}\right].
\end{equation}
Similar comments apply when $\omega\leq 0$, for which $B\leq0$, and $\theta^\prime\leq0$ for all $y\in[-1,1]$. If $B=0$, we either have $s=0$ or $\theta$=constant almost everywhere, compatible with the definition of an OR solution (unless $\omega=0$, and $(s,\theta)=(1,0)$, which is not an OR solution). Conversely, an OR solution, by definition, has $B=0$ since polydomain structures correspond to piecewise constant $\theta$-profiles. In other words, if $\omega\neq0$, OR solutions exist if and only if $B=0$. If $B\neq 0$, then OR solutions are necessarily disallowed because a non-zero value of $B$ implies that $s\neq 0$ on $\Omega$. The following results show that the choice of $B$ is in turn dictated by $\omega$, or the Dirichlet boundary conditions, and this sheds beautiful insight into how the boundary datum manifests in the multiplicity and regularity of solutions. 
In what follows, we let $\epsilon:=\frac{1}{L^*}$, so that $\epsilon \propto L^2$ where $L$ is the physical channel width.

Note that \eqref{eq:s-no-flow} and \eqref{eq:theta-no-flow} are the Euler-Lagrange equations of the following energy, 
\begin{align}
    F_{LG}[s,\theta]=\int_{\Omega}\left(\frac{( s^\prime)^2}{4}+s^2(\theta^\prime)^2\right)+\frac{\eps s^2}{4}\left(\frac{s^2}{2}-1\right)~\mathrm{d}y, \label{eq:s-energy}
\end{align}
but we only consider $(s, \theta) \in W^{1,2}\left(\Omega; \mathbb{R} \right)$ and focus on smooth, classical solutions of (\ref{eq:s-no-flow}) and (\ref{eq:theta-no-flow}), subject to the boundary conditions in (\ref{eq:s-bc})-(\ref{eq:theta-bc}), and not OR solutions. 
We first prove that  OR solutions only exist for the special values, $\omega = \pm \frac{1}{4}$, 
in the $\Qvec$-framework. 
If $\omega=\pm \frac{1}{4}$, then $B$ can be either zero or non-zero for different solution branches, especially for small values of $\epsilon$ that admit multiple solution branches. Once the correspondence between $\omega$, $B$ and OR solutions is established in the $\Qvec$-framework, we proceed to prove several qualitative properties of the corresponding $(s, \theta)$-profiles which are of independent interest, followed by asymptotics and numerical experiments (also see supplementary material). 

\begin{theorem}\label{thm:uniqueness+OR}
For all $\eps\geq0$, there exists a minimizer of the energy \eqref{eq:q-energy}, in the admissible space
\begin{multline}
\mathcal{A} = \left\{ \Qvec \in W^{1,2}\left([-1, 1]; S_2 \right); Q_{11}(\pm1)  = \frac{\cos(2\omega \pi)}{2},\right.\\
 \left.Q_{12}(-1)=-\frac{\sin 2\omega \pi}{2}, Q_{12}(1) = \frac{\sin 2\omega \pi}{2} \right\}.\label{eq:admissible_space}
\end{multline}
Moreover, the system (\ref{eq:Q-EL-eqnts}) admits an analytic solution for all $\epsilon \geq 0$, in $\mathcal{A}$. OR solutions only exist for $\omega=\pm\frac{1}{4}$ in (\ref{eq:Q-bcs}). 
\end{theorem}

\begin{proof}
The existence of an energy minimizer for \eqref{eq:q-energy} in $\mathcal{A}$, is immediate from the direct methods in the calculus of variations, for all $\epsilon$ and $\omega$, and the minimizer is a classical solution of the associated Euler-Lagrange equations (\ref{eq:Q-EL-eqnts}), for all $\epsilon$ and $\omega$. In fact, using standard arguments in elliptic regularity, one can show that all solutions of the system \eqref{eq:Q-EL-eqnts} are analytic \cite{brezis_bethuel_helein1993}. 

The key observation is
\begin{equation*}
    \left(Q_{12}^\prime Q_{11}-Q_{11}^\prime Q_{12}\right)^\prime=Q_{12}^{\prime\prime}Q_{11}+Q_{12}^\prime Q_{11}^\prime -Q_{12}^\prime Q_{11}^\prime-Q_{12}Q_{11}^{\prime\prime}=0,
\end{equation*}
and hence, $Q_{12}^\prime Q_{11}-Q_{11}^\prime Q_{12}$ is a constant. In fact, using \eqref{eq:s-theta-relationship}, we see that
\begin{equation*}
    (s^2\theta^\prime)^\prime=2(Q_{12}^{\prime\prime} Q_{11}-Q_{11}^{\prime\prime} Q_{12})=0 \implies 
    s^2\theta^\prime=2(Q_{12}^\prime Q_{11}-Q_{11}^\prime Q_{12})=B,\label{eq:s-theta-Q-B}
\end{equation*}
where $B$ is as in \eqref{eq:theta-eqtn}. 
Now let $B=0$ (so that OR solutions are possible), then
\begin{equation}
    Q_{12}^\prime Q_{11}=Q_{11}^\prime Q_{12}\textrm{ for all $y\in[-1,1]$}.\label{eq:Q-constraint}
\end{equation}
There are two obvious solutions of (\ref{eq:Q-constraint}) i.e. $Q_{11} \equiv 0$ (i.e., $\omega=\pm\frac{1}{4}$), or $Q_{12}\equiv 0$ (i.e., $\omega=0,\pm\frac{1}{2}$), everywhere on $\Omega$. For the case $Q_{12}\equiv0$ and $\omega=\pm\frac{1}{2}$, the Euler-Lagrange equations for $\Qvec$ reduce to
\begin{equation}
    \begin{cases}
    &Q_{11}^{\prime\prime}=\epsilon Q_{11}(4Q_{11}^2-1),\\
    &Q_{11}(-1)= -\frac{1}{2},\;Q_{11}(1)= -\frac{1}{2}.
    \end{cases}\label{eq:Q12-OR-equation}
\end{equation}
This is essentially the ODE considered in equation (20) of \cite{lamy}. Applying the arguments in Lemma 5.4 of \cite{lamy}, the solution $Q_{11}$ of \eqref{eq:Q12-OR-equation} must satisfy $Q^\prime_{11}(-1)=0$, or $Q^\prime_{11}$ is always positive. However, the latter is not possible since we have symmetric boundary conditions. Hence, when $\omega=\pm\frac{1}{2}$, the unique solution to \eqref{eq:Q12-OR-equation} is the constant solution $(Q_{11},Q_{12})=(-\frac{1}{2}$,0). 
This corresponds to $s=1$ everywhere in $\Omega$, which is not an OR solution. 
The same arguments apply to the case $Q_{12} \equiv 0$ and $\omega=0$. In this case the boundary conditions are $Q_{11}(\pm1)=\frac{1}{2}$, and the corresponding $(s, \theta)$ solution is simply, $(s,\theta)=(1,0)$, which is again not an OR solution.

When $Q_{11}\equiv0$ ($\omega=\pm\frac{1}{4}$), the $\Qvec$ system becomes 
\begin{equation}
    \begin{cases}
    &Q_{12}^{\prime\prime}=\epsilon Q_{12}(4Q_{12}^2-1),\\
    &Q_{12}(-1)=-\frac{1}{2},\;Q_{12}(1)=\frac{1}{2}.
    \end{cases}\label{eq:Q11-OR-equation}
\end{equation}
Applying the arguments in Lemma 5.4 of \cite{lamy}, we see \eqref{eq:Q11-OR-equation} has a unique solution which is odd and increasing, with a single zero at $y=0$ - the centre of the channel. This is an OR solution, since $Q_{11} = 0$ implies that $\theta$ is constant on either side of $y=0$.

It remains to show that there are no solutions $(Q_{11},Q_{12})$ of \eqref{eq:Q-EL-eqnts}, which satisfy \eqref{eq:Q-constraint}, other than the possibilities considered above.  
To this end, we assume that we have non-trivial solutions, $Q_{11}$ and $Q_{12}$ such that \eqref{eq:Q-constraint} holds. We recall that all solution pairs, $(Q_{11}, Q_{12})$ of (\ref{eq:Q-EL-eqnts}) are analytic and hence, can only have zeroes at isolated interior points of $\Omega = [-1,1]$. This means that there exists a finite number of intervals $(-1,y_1), \ldots,(y_n,1)$, such that $Q_{11}\neq0$ and $Q_{12}\neq0$ in the interior of these intervals, whilst either $Q_{11}(y_i)$, $Q_{12}(y_i)$, or both, equal zero at each intervals end-points. 
We then have that
\begin{equation*}
    \frac{Q_{12}^\prime}{Q_{12}}=\frac{Q_{11}^\prime}{Q_{11}}\implies |Q_{11}|=c_i|Q_{12}|\textrm{ for }
    y\in(y_{i-1},y_i)
\end{equation*}
for constants $c_i>0$ and $i=1,\ldots, n$. Therefore, there exists an interval, $(y_{i-1}, y_i)$, for which $Q_{11}$ and $Q_{12}$ have the same, or opposite signs. Assume without loss of generality (W.L.O.G.) $Q_{11}$ and $Q_{12}$ have the same sign, 
then the analytic function
\[
f(y) := Q_{11}(y) - c_i Q_{12}(y) = 0, \textrm{ for } y\in (y_{i-1},y_i).
\]
Therefore, $f(y) = 0$ for all $y\in [-1, 1]$. 
Evaluating at $y=\pm 1$, we have
\begin{equation*}
    \cos(2\omega\pi)=-\sin(2\omega\pi)c_i \textrm{  and }\cos(2\omega\pi)=\sin(2\omega\pi)c_i,
\end{equation*}
and this is only possible if $\cos(2\omega\pi)=0$ and $\sin(2\omega\pi)c_i=0$, which implies $\omega=\pm\frac{1}{4}$ and $c_i=0$. 
Hence, there are only three possibilities for $\omega = 0, \pm\frac{1}{4}, \pm \frac{1}{2}$ that are consistent with \eqref{eq:Q-constraint}, of which OR solutions are only compatible with $\omega = \pm \frac{1}{4}$. 
\end{proof}

In what follows, we consider the solution profiles, $(s,\theta)$ of \eqref{eq:s-no-flow} and \eqref{eq:theta-no-flow}, from which we can construct a solution of the system \eqref{eq:Q-EL-eqnts}, using the definitions \eqref{eq:Q-components}. 
 The first proposition below is adapted from \cite{majumdar-2010-article}, although some additional work is needed to deal with the positivity of $s$; see the supplementary material.

\begin{theorem}\label{thm:maximum_principle}
(Maximum Principle) Let $s$ and $\theta$ be  
solutions of \eqref{eq:s-no-flow} and \eqref{eq:theta-no-flow}, where $s$ is at least $C^2$ and $\theta$ is at least $C^1$, then
\begin{equation}
     0< s\leq 1\quad\forall y\in[-1,1]. 
\end{equation}
\end{theorem}

For the next batch of results, we omit the case $B=0$ and focus on the $(s, \theta)$-profiles of non OR-solutions, which are necessarily smooth.
We exploit this fact to prove that there exists a unique solution pair, $(s, \theta)$ of (\ref{eq:EL-eqtns}), 
such that $s$ has a symmetric even profile about $y=0$, for every $B\neq0$.

\begin{theorem} \label{thm:symmetry}

Any non-constant and non-OR solution, $s$, of the Euler-Lagrange equations \eqref{eq:EL-eqtns}, has a single critical point which is necessarily a non-trivial global minimum at some $y^*\in (-1, 1)$. 
\end{theorem}

\begin{proof}
For clarity, we denote a specific solution of \eqref{eq:s-no-flow} and \eqref{eq:theta-no-flow}, by $(s_{sol},\theta_{sol})$ in this proof.
Recall that for non-OR solutions, we necessarily have $B=\theta^\prime(\pm1)\neq0$ and $s\neq0$ anywhere. 
Using the definition of $B$ in (\ref{eq:EL-eqtns}), we have
\begin{equation}
s^{\prime\prime}=\frac{4B^2}{s^3}+\epsilon(s^3-s). \label{eq:An-1}
\end{equation}
The right hand side of \eqref{eq:An-1} is well-defined and continuous for $s\in(0,1]$, and as such, a solution, $s_{sol}$, will be $C^2$. In fact, the right hand side of \eqref{eq:An-1} is smooth, hence any solution, $s_{sol}$, will be smooth.
The boundary conditions, $s\left(\pm 1 \right) = 1$, imply that a non-trivial solution has $s_{sol}^\prime(y^*)=0$ for some $y^*\in[-1,1]$, where $s^\prime$ is defined as, 
\begin{equation}
    s^\prime=\pm\sqrt{\left(-4B^2s^{-2}+\epsilon\left(\frac{s^4}{2}-s^2\right)+J\right)}.\label{eq:s-prime}
\end{equation}
Here, $A$ is a constant of integration and  $J=4B^2+\frac{\eps}{2}+s^\prime(\pm1)^2$, hence, we must have
\begin{equation}
   J\geq 4B^2+\frac{\epsilon}{2}. \label{eq:A-constraint-1}
\end{equation}

Since $s^\prime$ is defined in terms of $s$ and not $y$, solutions of $s^\prime=0$ give us the extrema of a solution $s_{sol}$ (i.e., maxima or minima), rather than the location of the critical points on the $y$-axis.  The condition $s^\prime=0$ is equivalent to 
\begin{equation}
    J=4B^2 s^{-2}-\epsilon\left(\frac{s^4}{2}-s^2\right). \label{eq:2}
\end{equation}
Clearly if $\epsilon=0$, we can only have one extremum, namely $s=\sqrt{\frac{4B^2}{J}}$, which in view of the boundary conditions and maximum principle, must be a minimum. For $\epsilon>0$, solving \eqref{eq:2} is equivalent to computing the roots of $f(s)=0$ where
\begin{equation}
    f(s):=s^6-2s^4+\frac{2J}{\epsilon}s^2-\frac{8B^2}{\epsilon} \label{eq:f}.
\end{equation}

Firstly, note that $f$ has a root for $s\in(0,1]$, since $f(0)=\frac{-8B^2}{\epsilon}<0$ and $f(1)=-1+\frac{2J}{\epsilon}-\frac{8B^2}{\epsilon}\geq 0$, by \eqref{eq:A-constraint-1}.
Differentiating \eqref{eq:f}, we obtain
\begin{equation*}
    \frac{df}{ds}(s)=6s^5-8s^3+\frac{4J}{\epsilon}s \label{f prime},
\end{equation*}
and the critical points of $f$ are given by
\begin{equation}
    s=0,\;s_\pm=\sqrt{\frac{8\pm\sqrt{64-\frac{96J}{\epsilon}}}{12}},\label{eq:critical-points}
\end{equation}
provided that $A\leq\frac{2}{3}\epsilon$. There are now three cases to consider.

Case 1: If $J>\frac{2}{3}\epsilon$, $f(s)$ has one critical point at $s=0$, which is a negative global minimum. Hence, $f$ has one root in the range, $s\in(0,1]$. 

Case 2: Let $J=\frac{2}{3}\epsilon$, so that the two critical points $s_\pm$ coincide. The point $s=0$ is still a minimum of $f(s)$ and the coefficient of $s^6$ is positive (so $f\to\infty$ as $s\to\infty$), so we deduce that $s_\pm$ is a stationary point of inflection (this can be checked via direct computation).
So again, $f$ has one root for $s\in(0,1]$. 

Case 3: Finally, let $J<\frac{2}{3}\epsilon$, so that $s_\pm$ are distinct critical points of $f$.  
The point, $s=0$, is still a minimum of $f(s)$ and the coefficient of $s^6$ is positive, so that there are two possibilities: (a) $s_{\pm}$ are distinct saddle points, and since $f$ is increasing for $s>0$, we see $f$ has a single root for $s\in (0,1]$, or (b) $s_-$ is a local maximum and $s_+$ is a local minimum of $f(s)$. In the latter case, $s=0$ is still a global minimum for $f(s)$, because $f(s_+)>f(0)$. 
Using this information, we can produce a sketch of $f(s)$ (shown in Figure \ref{fig:case-3}), and there are 5 cases to consider for the number of roots of $f$. 

In cases (i) and (v) of Figure \ref{fig:case-3}, $f$ has only one root for $s\in (0, 1]$. Next, in order for the derivative $s_{sol}^\prime$ to be real, the term under the square root in \eqref{eq:s-prime}, has to be non-negative.  This requires that $f(s)\geq 0$ for all $s\in[c,1]$, for some $c>0$. Applying this argument to cases (ii) and (iii) in Figure \ref{fig:case-3} by omitting regions with $f(s)<0$, we deduce that $f$ has a single non-trivial root for $s\in (0,1]$.

For case (iv), we have two distinct roots in an interval such that $f(s)\geq 0$, one of which is $s_+$, and the other root is labelled as $s_1$. 
Recalling that $s_+$ is also a solution of $f^\prime(s)=0$, we deduce that $s_+$ is a repeated root of $f$. Then, $f$ can be factorised as:
\begin{align}
    f(s)&=(s-s_+)^2(s+s_+)^2(s-s_1)(s+s_1)\nonumber\\
    &=s^6-(2s_+^2+s_1^2)s^4+(s_+^4+2s_1^2s_+^2)s^2-s_1^2s_+^4\label{eq:s_+-repeated}.
\end{align}
Comparing the coefficient of $s^4$ and $s^0$ in \eqref{eq:f}, with \eqref{eq:s_+-repeated}, we have
$s_1^2=2(1-s_+^2)$ and $s_1^2=\frac{8B^2}{\epsilon s_+^4}$, which implies
\begin{equation}
    4B^2+\epsilon s_+^4(s_+^2-1)=0.\label{eq:s_+-repeated2}
\end{equation}
Comparing \eqref{eq:An-1} with \eqref{eq:s_+-repeated2}, we deduce that, 
$s^{\prime\prime}(s_+)=0$. By the uniqueness theory for Cauchy problems, this implies that $s_{sol} \equiv s_+$, which is inadmissible and this case is excluded.
\begin{figure}[ht]
    \centering
    \begin{minipage}{0.29\textwidth}
        \centering
        \includegraphics[width=1.0\textwidth]{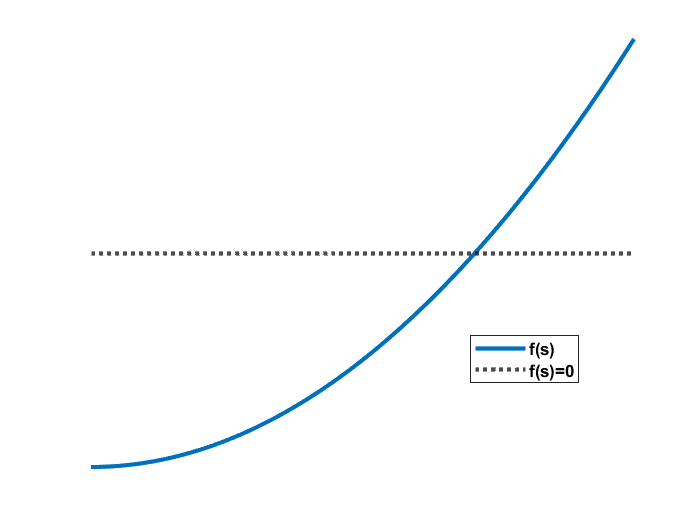}\\
        \ \textrm{Case 1}
    \end{minipage}
    \begin{minipage}{0.32\textwidth}
        \centering
        \includegraphics[width=1.0\textwidth]{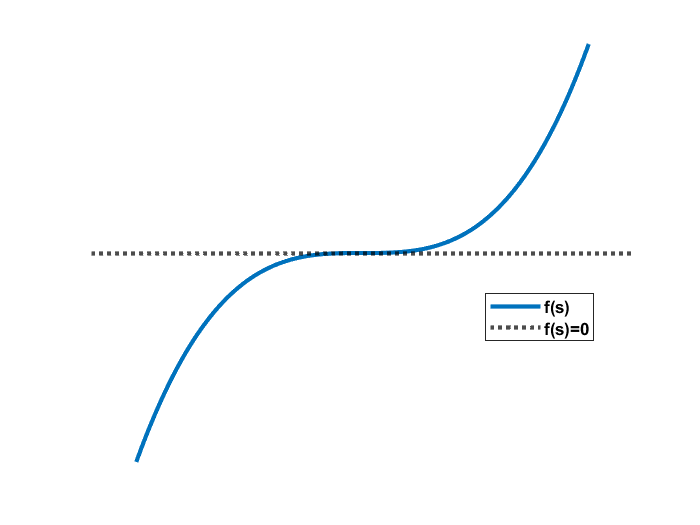}\\
        \ \textrm{Case 2}
    \end{minipage}
    \begin{minipage}{0.37\textwidth}
        \centering
        \includegraphics[width=1.0\textwidth]{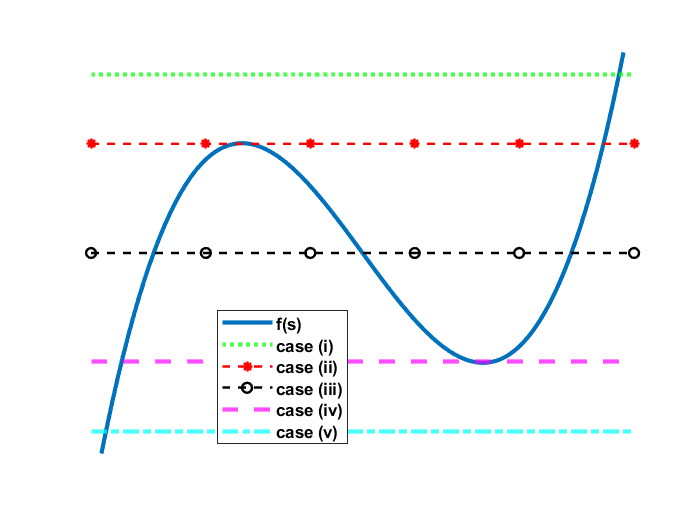}\\
        \ \textrm{Case 3 (b)}
    \end{minipage}
    \caption{The horizontal lines represent $f(s)=0$.}
    \label{fig:case-3}
\end{figure}

In cases 1, 2 and 3, we have demonstrated that 
$s_{sol}$ has a unique positive critical value, which must be the minimum value. 
The unique minimum value is attained at a unique interior point (if there were two interior minima at say $y^*$ and $y^{**}$, a non-constant solution would exhibit a local maximum between the two minima, which is excluded by a unique critical value for $s_{sol}$). This completes the proof. 
\end{proof}

\begin{theorem}\label{thm:uniqueness}
For a given 
$B=\theta^\prime(\pm1)\neq0$, 
the system \eqref{eq:EL-eqtns}, subject to the boundary conditions \eqref{eq:BCs}, admits a unique solution for a fixed $\eps$ and $\omega$. Hence, for any value of $\omega$ that does not permit $OR$ solutions, the system \eqref{eq:EL-eqtns} always has a unique solution. 
\end{theorem}
\begin{proof}
    Recall, for $\omega\neq0$, OR solutions exist if and only if $B=0$. When $\omega=0$, \eqref{eq:theta-no-flow} implies we must have $B=0$, the proof of Theorem \ref{thm:maximum_principle} (see supplementary material) then shows the unique solution in $W^{1,2}$ is $(s,\theta)=(1,0)$. 
        For $B\neq 0$, the system \eqref{eq:EL-eqtns} can be written as
\begin{subequations}\label{eq:unique-system}
    \begin{align}
        & s^{\prime\prime}=\frac{4B^2}{s^3}+\eps s(s^2-1),\label{eq:s-eqtn-2}\\
        & s^2\theta^\prime=B.\label{eq:B-non-zero}
    \end{align}
\end{subequations}
Throughout this proof we take $B>0$, so that $s\neq0$ and hence, the right hand side of \eqref{eq:s-eqtn-2} is analytic. The case $B<0$ can be tackled in the same manner.

In the first step, we show that \eqref{eq:unique-system} has a unique solution for fixed $B$, $\eps$ and $\omega$. Assume for contradiction that $(s_1,\theta_1)$ and $(s_2,\theta_2)$ are distinct solutions pairs of \eqref{eq:unique-system}, which satisfy \eqref{eq:BCs}. As such, they must have distinct derivatives at $y=-1$ (otherwise they would satisfy the same Cauchy problem). Suppose W.L.O.G. 
\begin{equation}
    s_1^\prime(-1)<s_2^\prime(-1)\leq 0.\label{eq:derivative-condition-1}
\end{equation}
Since $s_1(1)=s_2(1)=1$, there exists $y_0=\textrm{min}\{y>-1:s_1(y_0)=s_2(y_0):=s_0\}$. Therefore, $s_1<s_2$ for all $y\in(-1,y_0)$. Further, since $s_1$ and $s_2$ have one non-trivial global minimum (Theorem \ref{thm:symmetry}), there are four possibilities for the location of $y_0$: (i) Case I: $y_0 = 1$; (ii) Case II: $y_0 < \min \left\{\alpha, \beta\right\}$ where $s_1$ attains its unique minimum at $y=\alpha$ and $s_2$ attains its unique minimum at $y=\beta$; (iii) Case III: $\alpha \leq y_0 \leq \beta$, or $\beta \leq y_0 \leq \alpha$; and (iv) Case IV: $y_0> \max\left\{\alpha, \beta \right\}$. 
In case I, $s_1<s_2$ implies $\theta^\prime_1>\theta_2^\prime$ for all $y\in(-1,1)$, since both solution pairs satisfy \eqref{eq:B-non-zero}. Hence, $\theta_1(y)-\theta_2(y)$ is increasing, and cannot vanish at $y=1$, contradicting the boundary condition at $y=1$.

For Case II, we have
\[
s_2^\prime(y_0)\leq s_1^\prime(y_0) < 0
\]
so that
\begin{equation*}
 (s_2^\prime(-1))^2-(s_2^\prime(y_0))^2 < (s_1^\prime(-1))^2-(s_1^\prime(y_0))^2.
\end{equation*}
Using \eqref{eq:s-prime}, this is equivalent to
\begin{multline*}
 -4B^2-\frac{\eps}{2}+J_2-\left(- \frac{4B^2}{s_0^2} +\eps s_0^2\left(\frac{s_0^2}{2}-1\right)+J_2\right) <\\ -4B^2-\frac{\eps}{2}+J_1-\left( -\frac{4B^2}{s_0^2} + \eps s_0^2\left(\frac{s_0^2}{2}-1\right)+J_1\right),
\end{multline*}
where $J_1$ and $J_2$ are constants of integration associated with $s_1$ and $s_2$ respectively, and may not be equal. However, the left and right hand sides are in fact equal, yielding the desired contradiction.

For Cases III and IV, there must exist another point of intersection, $y=y_1 \in (\max\left\{\alpha, \beta\right\}, 1]$, such that
\[
\left(s_1 - s_2 \right)(y_1) = 0;\quad \left(s_1 - s_2 \right)^{\prime}(y_1)< 0
\]
and
\[
0 < s_1^{\prime}(y_1) \leq s_2^{\prime}(y_1).
\]
In this case, we can use
\begin{equation*}
 (s_2^\prime(-1))^2-(s_2^\prime(y_1))^2 < (s_1^\prime(-1))^2-(s_1^\prime(y_1))^2
\end{equation*}
to get the desired contradiction. 
We therefore conclude that for fixed $B$, $\eps$ and $\omega$, the solution of (\ref{eq:EL-eqtns}) is unique.

Next, we show the constant $B$ is unique for fixed $\eps$ and $\omega$. We assume that there exist two distinct solution pairs, $(s_1,\theta_1)$ and $(s_2,\theta_2)$, which by the first part of the proof, are the unique solutions of
\begin{align*}
    & s_1^{\prime\prime}=\frac{4B^2_1}{s_1^3}+\eps s_1(s_1^2-1)
    ,\quad s_2^{\prime\prime}=\frac{4B^2_2}{s_2^3}+\eps s_2(s_2^2-1)
\end{align*}
and $
 s_1^2\theta^\prime_1=B_1,s_2^2\theta^\prime_2=B_2,$
respectively, subject to \eqref{eq:BCs}, for the same value of $\omega$. 
Let $0<B_1\leq B_2$. Using a change of variable $u_k = 1 - s_k \in [0, 1)$, for $k=1,2$ so that $u_k(\pm 1) = 0$, we can use the method of sub- and supersolutions to deduce that 
\begin{equation}
    s_2\leq s_1 \textrm{ for all }y\in[-1,1].
\end{equation}
This implies 
\begin{equation}
    \theta^\prime_1=\frac{B_1}{s_1^2}\leq\frac{B_2}{s_2^2}=\theta^\prime_2  \quad \forall y\in [-1,1].
\end{equation}
If $\theta_1^\prime<\theta_2^\prime$ anywhere, then $\theta_1(1)=\omega\pi$ does not hold, hence we must have equality i.e., $\theta_1^\prime=\theta^\prime_2$. It therefore follows that $B_1 s_2^2= B_2 s_1^2$, but the boundary conditions necessitate that $B_1=B_2:=B$ and hence, $s_1=s_2:=s$. 
Finally, integrating $\theta_1^\prime=B/s^2$, it follows that $\theta_1$ is unique and is given by
\begin{equation}
    \theta_1(y)=\omega\pi-\int_y^1\frac{B}{s^2}~\mathrm{d}y, \textrm{ where }B=2\omega\pi \left(\int^{1}_{-1}\frac{1}{s^2}~\mathrm{d}y\right)^{-1}.
\end{equation}
The preceding arguments show that 
$\theta_1=\theta_2$ and the proof is complete. 
\end{proof}

\begin{theorem} \label{thm:symmetryagain}
For $B=\theta^\prime(\pm1)\neq0$, 
the unique solution, $\left(s, \theta\right)$ of (\ref{eq:EL-eqtns}), has the following symmetry properties:
\[
s(y) = s(-y)
\\  \qquad 
\theta(y) = -\theta(-y) 
\]
for all $y \in [-1, 1]$. Then $s$ has a unique non-trivial minimum at $y=0$.
\end{theorem}
\begin{proof}
It can be readily checked that for $B\neq0$ 
, the system of equations (\ref{eq:EL-eqtns}) admits a solution pair, $(s, \theta)$ such that $s$ is even, and $\theta$ is odd for $y \in [-1, 1]$, compatible with the boundary conditions. Combining this observation with the uniqueness result for $B\neq0$, 
the conclusion of the theorem follows.
\end{proof}

The preceding results apply to non OR-solutions. OR solution-branches have been studied in detail, in a one-dimensional setting, in the $\Qvec$-framework \cite{lamy}. 
Using the arguments in \cite{lamy}, one can prove that for $\omega=\pm\frac{ 1}{4}$, OR solutions exist for all $\eps\geq0$ and are globally stable as $\epsilon \to 0$, but lose stability as $\epsilon$ increases. In particular, non-OR solutions emerge as $\epsilon$ increases, for $\omega = \pm \frac{1}{4}$, and these non-OR solutions do not have polydomain structures. More precisely, we can explicitly compute limiting profiles in the $\eps \to 0$ and $\eps \to \infty$ limits. These calculations (which yield good insight into the more complex cases of non-constant velocity and pressure for passive and active nematodynamics considered next) can be found in the supplementary material (\cite{fang2019},\cite{han2020siap},\cite{braides} are associated new references appearing in the supplementary material).

\section{Passive and Active flows}
\label{sec:passive+active}
In this section, we compute asymptotic expansions for OR-type solutions of the system \eqref{eq:flow-system}, in the $L^*\to 0$ limit ($\epsilon \to \infty$ limit) relevant to micron-scale channels. We consider conventional passive nematodynamics and active nematodynamics (with additional stresses generated by internal activity), and generic scenarios with non-constant velocity and pressure.
We follow the asymptotic methods in \cite{CM} to construct OR-type solutions, strongly reminiscent of chevron patterns seen in experiments \cite{agha, copar}. Recall an OR-type solution is simply a solution of \eqref{eq:flow-system} with a non-empty nodal set for the scalar order parameter, such that $\theta$ has a planar jump discontinuity at the zeroes of $s$. 
Unlike OR solutions, OR-type solutions need not have polydomains with constant $\theta$-profiles. 




\subsection{Asymptotics for OR-type solutions in passive nematodynamics, in the \texorpdfstring{$L^*\to0$}{Lg} limit}
\label{sec:passive-asymptotics}
Consider the system, \eqref{eq:flow-system}, in the $L^*\to0$ limit. Motivated by the results of section \ref{sec:constant-pressure-flow}, and for simplicity, we assume $s$ attains a single minimum at $y=0$, $s$ is even and $\theta$ is odd, throughout this section.
The first step is to calculate the flow gradient $u_y$. We multiply \eqref{eq:theta-eqtn} by \textcolor{red}{$s$} so that 
\begin{equation}
    (s^2\theta_y)_{y}=\frac{s^2}{2}u_y.\label{eq:sec3-1}
\end{equation}
Substituting $(s^2\theta_y)_y$ from \eqref{eq:sec3-1} into \eqref{eq:u-eqtn}, we obtain
\begin{equation}
    \left(u_y+\frac{L_2}{2}s^2u_y\right)_y=p_x. \label{eq:sec3-2}
\end{equation}
Both sides of \eqref{eq:sec3-2} equal a constant, since the left hand side is independent of $x$, and $p_x$ is independent of $y$. Integrating \eqref{eq:sec3-2},
we find
\begin{equation}
    u_y=\frac{p_xy}{g(s)}+\frac{B_0}{g(s)},\label{eq:u_y}
\end{equation}
where $B_0$ is another constant and
\begin{equation}
    g(s)=1+\frac{L_2}{2}s^2>0,\;\forall s\in\mathbb{R}. \label{eq:g}
\end{equation}

Integrating \eqref{eq:u_y}, we have
\begin{equation}
    u(y)=\int^y_{-1}\frac{p_x Y}{g(s(Y))}+\frac{B_0}{g(s(Y))}~\mathrm{d}Y,\label{eq:pois-u-1}
\end{equation}
since $u(-1)=0$ from \eqref{eq:poiseuille}. Using the no-slip condition, $u(1)=0$ and the fact that $\int^1_{-1}\frac{Y}{g(s(Y))}~\mathrm{d}Y=0$, we obtain $B_0=0$ so that the flow velocity is given by
$   u(y)=\int^y_{-1}\frac{p_x Y}{g(s(Y))}~\mathrm{d}Y,$
and the corresponding velocity gradient is
\begin{equation}
    u_y(y)=\frac{p_xy}{g(s)}.\label{eq:poiseuille-gradient}
\end{equation}


Following the method in \cite{CM}, we  assume 
\begin{subequations}
\begin{align}
& s(y)=S(y)+IS(\lambda)+\mathcal{O}(L^*),\label{eq:s-expansion}\\
& \theta(y)=\Theta(y)+I\Theta(\lambda)+\mathcal{O}(L^*)\label{eq:theta-expansion},
\end{align}
\end{subequations}
where $S,\Theta$ represent the outer solutions away from the jump point at $y=0$, $IS,I\Theta$ represent the inner solutions around $y=0$, and $\lambda$ is our inner variable. Substituting these expansions into \eqref{eq:s-eqtn} and \eqref{eq:theta-eqtn} yields 
\begin{subequations}
\begin{align}
& L^*S_{yy}+L^*IS_{yy}=4L^*(S+IS)(\Theta_{y}+I\Theta_y)^2+(S+IS)((S+IS)^2-1),\label{eq:inner-s-eqtn}\\
& (S+IS)(\Theta_{yy}+I\Theta_{yy})=\frac{1}{2}(S+IS)u_y(y)-2(S_{y}+IS_{y})(\Theta_y+I\Theta_y)\label{eq:inner-theta-eqtn}.
\end{align}
\end{subequations}
It is clear that \eqref{eq:inner-s-eqtn} is a singular problem in the $L^* \to 0$ limit, and as such we rescale $y$ and set
\begin{equation}
    \lambda=\frac{y}{\sqrt{L^*}},\label{eq:dominant-balance}
\end{equation}
to be our inner variable. 

The outer solution is simply the solution of \eqref{eq:inner-s-eqtn} and \eqref{eq:inner-theta-eqtn}, away from $y=0$, for $L^*=0$ and when internal contributions are ignored.
In this case, \eqref{eq:inner-s-eqtn} reduces to
\begin{eqnarray}
 S(S^2-1)=0,
\end{eqnarray}
which implies
\begin{equation}
\label{eq:s-outer}
    S(y)=
    1,\quad\textrm{for $y\in[-1,0)\cap(0,1]$ }
\end{equation}
is the outer solution. Here we have ignored the trivial solution $S=0$, and $S=-1$, as these solutions do not satisfy the boundary conditions. 

Ignoring internal contributions, 
\eqref{eq:inner-theta-eqtn} reduces to 
\begin{equation}
    \Theta_{yy}(y)=\frac{1}{2}u_y(y)\quad\textrm{for $y\in[-1,0)\cap(0,1]$}\label{eq:outer-theta-eqtn}.
\end{equation}
From the above, $s=1$ for $y\in[-1,0)\cap(0,1]$, therefore, integrating \eqref{eq:poiseuille-gradient} 
and imposing the no-slip boundary conditions \eqref{eq:poiseuille}, we obtain
\begin{align}
    u(y)=\frac{p_x}{2+L_2}(y^2-1)\label{eq:u-asymptotic}.
\end{align}
We take $u(0)=-\frac{p_x}{2+L_2}$, consistent with the above expression.
Solving for $0<y\leq1$, we integrate \eqref{eq:outer-theta-eqtn} to obtain
\begin{align}
    &\Theta_{y}(y)=\int_0^y\frac{u_y(Y)}{2}~dY+\Theta_y(0+)\nonumber\\
    &\implies\Theta_{y}(y)=\frac{u(y)-u(0)}{2}+\Theta_y(0+)\label{eq:outer-theta_y}.
\end{align}
Similarly, for $-1\leq y<0$, integrating \eqref{eq:outer-theta-eqtn} yields
\begin{align}
    &\Theta_{y}(y)=\frac{u(y)-u(0)}{2}+\Theta_y(0-).\label{eq:outer-theta_y-2}
\end{align}

Since $\Theta_y(0\pm)$ is unknown, we enforce the following boundary conditions at $y=0$ to give us an explicitly computable expression
\begin{subequations} \label{eq:jump_k}
\begin{align}
    & \Theta(0+)=\omega\pi -\frac{k\pi}{2},\; k\in\mathbb{Z},\\
    & \Theta(0-)=-\omega\pi +\frac{k\pi}{2},\; k\in\mathbb{Z}.
\end{align}
\end{subequations}
We now justify this jump condition. In the case of constant flow and pressure, OR solutions jump by $\pm 2 \omega\pi$, but OR-type solutions could have different jump conditions across the domain walls, 
hence the inclusion of the $\frac{k\pi}{2}$ term (other jump terms are also possible). Substituting \eqref{eq:u-asymptotic} into \eqref{eq:outer-theta_y}, integrating, and imposing the boundary conditions, we have that
\begin{align}
   \Theta(y)=\frac{p_x}{(2+L_2)}\left(\frac{y^3}{6}-\frac{y}{6}\right)
   +\frac{k\pi}{2}(y-1)+\omega\pi\quad\textrm{for $y\in(0,1]$}.\label{eq:theta-outer-y-positve}
\end{align}
Analogously, \eqref{eq:outer-theta_y-2} yields
\begin{equation}
    \Theta(y)=\frac{p_x}{(2+L_2)}\left(\frac{y^3}{6}-\frac{y}{6}\right)+\frac{k\pi}{2}(y+1)
    -\omega\pi\quad\textrm{for $y\in[-1,0)$}.\label{eq:theta-outer-y-negative}
\end{equation}

We now compute the inner solution. Substituting the inner variable \eqref{eq:dominant-balance} into \eqref{eq:inner-s-eqtn} and \eqref{eq:inner-theta-eqtn}, 
they become
\begin{eqnarray*}
&& L^*S_{yy}+\ddot{IS}=4L^* (S+IS)\left(\Theta_{y}+\frac{\dot{I\Theta}}{\sqrt{L^*}} \right)^2+(S+IS)((S+IS)^2-1),\\
&& (S+IS)(L^*\Theta_{yy}+\ddot{I\Theta})=\frac{L^*}{2}(S+IS)u_y(\lambda \sqrt{L^*})-2L^*\left(S_{y}+\frac{\dot{IS}}{\sqrt{L^*}}\right)\left(\Theta_y+\frac{\dot{I\Theta}}{\sqrt{L^*}}\right),
\end{eqnarray*}
where $\dot{()}$ denotes differentiation w.r.t $\lambda$.
Letting $L^*\to0$, we have that the leading order equations are
\begin{subequations}
\begin{align}
& \ddot{IS}=4(S+IS)(\dot{I\Theta})^2+(S+IS)((S+IS)^2-1),\\
& (S+IS)\ddot{I\Theta}=
-2\dot{IS}\dot{I\Theta},
\end{align}
\end{subequations}
or equivalently, after recalling $S=1$,
\begin{eqnarray*}
&& \ddot{IS}=2IS+q_1(IS,\dot{I\Theta}), \quad
 \ddot{I\Theta}=
q_2(IS,I\dot{S},\dot{I\Theta},\ddot{I\Theta}),
\end{eqnarray*}
where $q_1,q_2$ represent the nonlinear terms of the equation. The linearised system is
\begin{subequations}
\begin{align}
& \ddot{IS}=2IS,\label{eq:s-linearised-2}\\
& \ddot{I\Theta}=0,\label{eq:theta-linearised-2}
\end{align}
\end{subequations}
subject to the boundary and matching conditions
\begin{subequations}
\begin{align}
& \underset{\lambda\to\pm\infty}{\textrm{lim}}IS(\lambda)=0,\;IS(0)=s_{min}-1,\label{eq:final-bcs-2}\\
& \underset{\lambda\to\pm\infty}{\textrm{lim}}I\Theta(\lambda)=0\label{eq:final-bcs-theta-2},
\end{align}
\end{subequations}
where $s_{min}\in[0,1]$, is the minimum value of $s$. We note that the second condition in \eqref{eq:final-bcs-2} ensures $s(0)=s_{min}$.
 Using the conditions (\ref{eq:final-bcs-2}), the solution of \eqref{eq:s-linearised-2} is
\begin{equation}
\label{eq:s-comp}
    s(y)=
    \begin{cases}
    1+(s_{min}-1)e^{-\sqrt{2}\frac{y}{\sqrt{L^*}}}&\quad\textrm{for $0\leq y\leq1$}\\
    1+(s_{min}-1)e^{\sqrt{2}\frac{y}{\sqrt{L^*}}}&\quad\textrm{for $-1\leq y\leq 0$}.
    \end{cases}
\end{equation}
With $IS$ determined, we calculate $I\Theta$. 
Solving (\ref{eq:theta-linearised-2}) subject to the limiting conditions \eqref{eq:final-bcs-theta-2}, it is clear that $I\Theta=0$. Hence,
\begin{align}\label{eq:theta-asymptotic}
    \theta(y)=\begin{cases}
    &\frac{p_x}{(2+L_2)}\left(\frac{y^3}{6}-\frac{y}{6}\right)+\frac{k\pi}{2}(y-1)+\omega\pi
    \quad\textrm{for $0<y\leq1$}\\
    &\frac{p_x}{(2+L_2)}\left(\frac{y^3}{6}-\frac{y}{6}\right)
   +\frac{k\pi}{2}(y+1)-\omega\pi
    \quad\textrm{for $-1\leq y<0$}.
    \end{cases}
\end{align}
The expressions, \eqref{eq:s-comp} and \eqref{eq:theta-asymptotic}, are consistent with our definition of an OR-type solution.

\subsection{Asymptotics for OR-type solutions in active nematodynamics, in the \texorpdfstring{$L^*\to0$}{Lg} limit}\label{sec:active}
Next, we consider an active nematic system in a channel geometry, i.e., a system that is constantly driven out of equilibrium by internal stresses and activity \cite{Giomo-2012}. There are three dependent variables to solve for: the concentration, $c$, of active particles, the fluid velocity $\mathbf{u}$, and the nematic order parameter $\Qvec$. The corresponding evolution equations are taken from  \cite{active-defects, Giomo_annihilation}, with additional \emph{active stresses} from the self-propelled motion of the active particles and the non-equilibrium intrinsic activity:
\begin{subequations}\label{eq:active_system}
\begin{align}
    & \frac{Dc}{Dt}=\nabla\cdot\left(\mathbf{D}\nabla c +\alpha_1 c^2 (\nabla\cdot\Qvec)\right),\label{eq:c_dimensional_eqtn}\\
    & \nabla\cdot\mathbf{u}=0,\quad\rho\frac{D\mathbf{u}}{Dt}=-\nabla p +\nabla\cdot(\mu(\nabla \mathbf{u}+(\nabla\mathbf{u})^T)+\tilde{\sigma}),\\
    & 
    \frac{D\Qvec}{Dt}=\lambda s \mathbf{W}+
    \zeta\Qvec-\Qvec\zeta+\frac{1}{\gamma}\mathbf{H},
\end{align}
\end{subequations}
where $\mathbf{W}$  
is the symmetric part of the velocity gradient tensor, $D_{ij}=D_0 \delta_{ij}+D_1 Q_{ij}$ is the anisotropic diffusion tensor ($D_0=(D_\parallel+D_\perp)/2$, $D_1=D_\parallel -D_\perp$ and $D_\parallel$ and $D_\perp$ are, respectively, the bare diffusion
coefficients along the parallel and perpendicular directions of the director field), $\alpha_1$ is an activity parameter, and $\lambda$ is the nematic alignment parameter,  which characterizes the relative dominance of the strain and the vorticity in affecting the alignment of particles with the flow \cite{review}. For $|\lambda| < 1$, the rotational part of the flow dominates, while for $|\lambda|>1$, the director will tend to align at a unique angle to the flow direction \cite{edwards_active}.
The value of $\lambda$ is also determined by the shape of the active particles \cite{giomi_sheared}. 
The stress tensor, $\tilde{\sigma}=\sigma^e+\sigma^a$ \cite{excitable},  is the sum of an elastic stress due to nematic elasticity 
\begin{equation}
    \sigma^e=-\lambda s \mathbf{H}+
    \Qvec\mathbf{H}-\mathbf{H}\Qvec,
\end{equation}
and an active stress defined by
\begin{equation}
    \sigma^a=\alpha_2 c^2\Qvec.
\end{equation}
Here $\alpha_2$ is a second activity parameter, which describes extensile (contractile) stresses exerted by the active particles when $\alpha_2<0$ ($\alpha_2>0$). $\mathbf{H}$, $\mu$, $\xi$, $p$ and $\rho$, are as introduced in \Cref{sec:theory}.

We again consider a one-dimensional static problem, with a unidirectional flow in the $x$ direction and take $\lambda=0$ for simplicity and in order to focus on the effect of other parameters relevant to this study. 
Then the evolution equations for $\mathbf{Q}$ are the same as those considered in the passive case, hence, making it easier to adapt the calculations in section \ref{sec:passive-asymptotics} and draw comparisons between the passive and active cases. 
The isotropic to nematic phase transition is driven by the concentration of active particles and as such, we take $A=\kappa(c^*-c)/2$ and $C=\kappa c$, 
where  $c^*=\sqrt{3\pi/2L^2}$ is the critical concentration at which this transition occurs \cite{Giomo-2012,active-defects}. As in the passive case, we work with $A<0$ i.e. with concentrations that favour nematic ordering.

The continuity equation \eqref{eq:c_dimensional_eqtn}, follows from the fact that the total number of active particles must remain constant \cite{Giomo-2012}. This is compatible with constant concentration, $c$, although solutions with constant concentration do not exist for $\alpha_1 \neq 0$. 
We consider the case of constant concentration $c$, which is not unreasonable for small values of $\alpha_1$ and certain solution types (see supplementary material for further details), and do not consider the concentration equation, \eqref{eq:c_dimensional_eqtn}, in this work.
We nondimensionalise the system as before, but additionally scale $c$ and $c^*$ by $L^{{-2}}$ (e.g, $c=L^{-2}\Tilde{c}$, where $\Tilde{c}$ is dimensionless). In terms of $\Qvec$, the evolution equations are given by
\begin{subequations}\label{eq:Q-active}
    \begin{align}
        & \frac{\partial Q_{11}}{\partial t}=u_y Q_{12}+ Q_{11,yy}+\frac{1}{L^*}Q_{11}(1-4(Q_{11}^2+Q_{12}^2)), \\
        & \frac{\partial Q_{12}}{\partial t}=-u_y Q_{11}+ Q_{12,yy}+\frac{1}{L^*}Q_{12}(1-4(Q_{11}^2+Q_{12}^2)), \\
        & L_1\frac{\partial u}{\partial t}=-p_x+u_{yy}+2L_2(Q_{11}Q_{12,yy}-Q_{12}Q_{11,yy})_y+\Gamma(Q_{12}c^2)_y,
    \end{align}
\end{subequations}
where 
$\Gamma=\frac{\alpha_2\gamma}{\kappa\mu\textcolor{red}{L^2}}\sqrt{-\frac{2A}{C}}$ is a measure of activity.
In the steady case, and in terms of $(s,\theta)$, the system \eqref{eq:Q-active} reduces to
\begin{subequations}
    \begin{align}
    & s_{yy}=4s\theta^2_y+\frac{s}{L^*}\left(s^2-1\right),\label{eq:active-s}\\
    & s\theta_{yy}=\frac{1}{2}su_y-2s_y\theta_y,\label{eq:active-theta}\\
    & u_{yy}=p_x-L_2(s^2\theta_y)_{yy}-\Gamma\left( \frac{c^2 
    s}{2}\sin(2\theta)\right)_y.\label{eq:active-u1}
    \end{align}
\end{subequations}
Regarding boundary conditions, we impose the same boundary conditions on $s$, $\theta$ and $u$, as in the passive case.

The equations, \eqref{eq:active-s} and \eqref{eq:active-theta}, are identical to the equations, \eqref{eq:s-eqtn} and \eqref{eq:theta-eqtn}, respectively. Hence, the asymptotics in subsection \ref{sec:passive-asymptotics} remain largely unchanged, with differences coming from \eqref{eq:active-u1}, due to the additional active stress. Skipping technical details which are analogous to those in \Cref{sec:passive-asymptotics}, we find the fluid velocity is given by
\begin{equation}
    u(y)=\int^y_{-1}\frac{2p_x Y-\Gamma c^2s(Y)\sin(2\theta(Y))}{2g(s(Y))}dY.\label{eq:u-active}
\end{equation}

Following methods in subsection \ref{sec:passive-asymptotics}, we pose asymptotic expansions as in \eqref{eq:s-expansion} and \eqref{eq:theta-expansion}, for $s$ and $\theta$ respectively in the $L^* \to 0$ limit, which yields \eqref{eq:inner-s-eqtn} and \eqref{eq:inner-theta-eqtn}.
In fact, the expression for $s$ is given by \eqref{eq:s-comp}, in the active case as well. 
For $\Theta$, we again solve \eqref{eq:outer-theta-eqtn}  and find an implicit representation as given below:
\begin{equation}
    \Theta(y)=
    \begin{cases}
    & \int_y^{1}\frac{u(0)-u(Y)}{2}dY
    +\left(\frac{k\pi}{2}-\int_0^1\frac{u(Y)-u(0)}{2}dY\right)(y-1)+\omega\pi,\;0<y\leq 1 \\
    & \int_{-1}^{y}\frac{u(Y)-u(0)}{2}dY
    +\left(\frac{k\pi}{2}-\int_{-1}^0\frac{u(Y)-u(0)}{2}dY\right)(y+1)-\omega\pi,\;-1\leq y<0
    \end{cases}\label{eq:theta-active-expansion}
\end{equation}
where $u(y)$ is given by \eqref{eq:u-active}. 
Moving to the inner solution $I\Theta$, we need to solve \eqref{eq:theta-linearised-2}, subject to the matching condition \eqref{eq:final-bcs-theta-2}.  
As before, we find $I\Theta=0$, and our composite expansion for $\theta$ is just the outer solution presented above.  
We deduce that OR-type solutions are still possible in an active setting, for the case $\lambda=0$.

We now consider a simple case for which  \eqref{eq:theta-active-expansion} can be solved explicitly. 
In \eqref{eq:u-active}, we assume $s=1$ and $\sin2\theta=1$ for $-1\leq y<0$, and $\sin(2\theta)=-1$ for $0<y\leq 1$ i.e., we assume an OR solution with $\theta=\mp\frac{\pi}{4}$ and $\omega=-\frac{1}{4}$. Under these assumptions, \eqref{eq:u-active} 
yields
\begin{equation}
    u(y)=
    \begin{cases}
    &\frac{p_x }{2+L_2}(y^2-1)+\frac{\Gamma c^2}{2+L_2}(y-1),\quad \textrm{for }0<y\leq 1 \\
    & \frac{p_x }{2+L_2}(y^2-1)-\frac{\Gamma c^2}{2+L_2}(y+1),\quad \textrm{for }-1\leq y< 0. 
    \end{cases}\label{eq:u-active-asymptotic}
\end{equation}
Substituting the above into \eqref{eq:theta-active-expansion}, we find
\begin{equation}\label{eq:theta-active-explicit}
    \theta(y)=
    \begin{cases}
    & \frac{p_x }{2+L_2}\left(\frac{y^3}{6}-\frac{y}{6}\right)+\frac{\Gamma c^2}{2+L_2}(\frac{y^2}{4}-\frac{y}{4})
    +\frac{k\pi}{2}(y-1)+\omega\pi,\quad\textrm{for $0<y\leq 1$}\\
    & \frac{p_x }{2+L_2}\left(\frac{y^3}{6}-\frac{y}{6}\right)-\frac{\Gamma c^2}{2+L_2}(\frac{y^2}{4}+\frac{y}{4})
    +\frac{k\pi}{2}(y+1)-\omega\pi\quad\textrm{for $-1\leq y<0$}.
    \end{cases}
\end{equation}
We expect (\ref{eq:u-active-asymptotic}) and (\ref{eq:theta-active-explicit}) to be good approximations to OR-type solutions with $\omega =-\frac{1}{4}$, in the limit of small $\Gamma$ (small activity)  and small pressure gradient, when the outer solution is well approximated by an OR solution.

\subsection{Numerical results}\label{sec:numerics}
We solve the dynamical systems \eqref{eq:Q-flow} and \eqref{eq:Q-active} with finite element methods, and all simulations are performed using the open-source package FEniCS~\cite{logg2012automated}. The details of the numerical methods are given in the supplementary material. 
In the numerical results that follow, we extract the $s$ profile from $\Qvec$, using \eqref{eq:s-theta-relationship}. 

\subsubsection{Passive flows}
We begin by investigating whether OR-type solutions exist for the passive system \eqref{eq:Q-flow} when $L^*$ is large (small $\epsilon$), that is, for small nano-scale channel domains. 
When $\omega=\pm\frac{1}{4}$ and $p_x=-1$, we find profiles which are small perturbations of the limiting OR solutions reported in the supplementary material, for large $L^*$ and $p_x = 0$, i.e., (2.7a), (2.7b) in the supplementary material 
 when $\omega=\pm\frac{1}{4}$ (see Fig. \ref{fig:L_infinty}).
We regard these profiles as being OR-type solutions although $s(0) \neq 0$ but $s(0)\ll 1$, as the director profile resembles a polydomain structure and $\theta$ jumps around $y=0$, to satisfy its boundary conditions. As $|p_x|$ increases, we lose this approximate zero in $s$, i.e., we lose the domain wall and $s \to 1$ almost everywhere. 

\begin{figure}[ht]
    \centering
    \begin{minipage}{1.0\textwidth}
        \centering
        \includegraphics[width=1\textwidth]{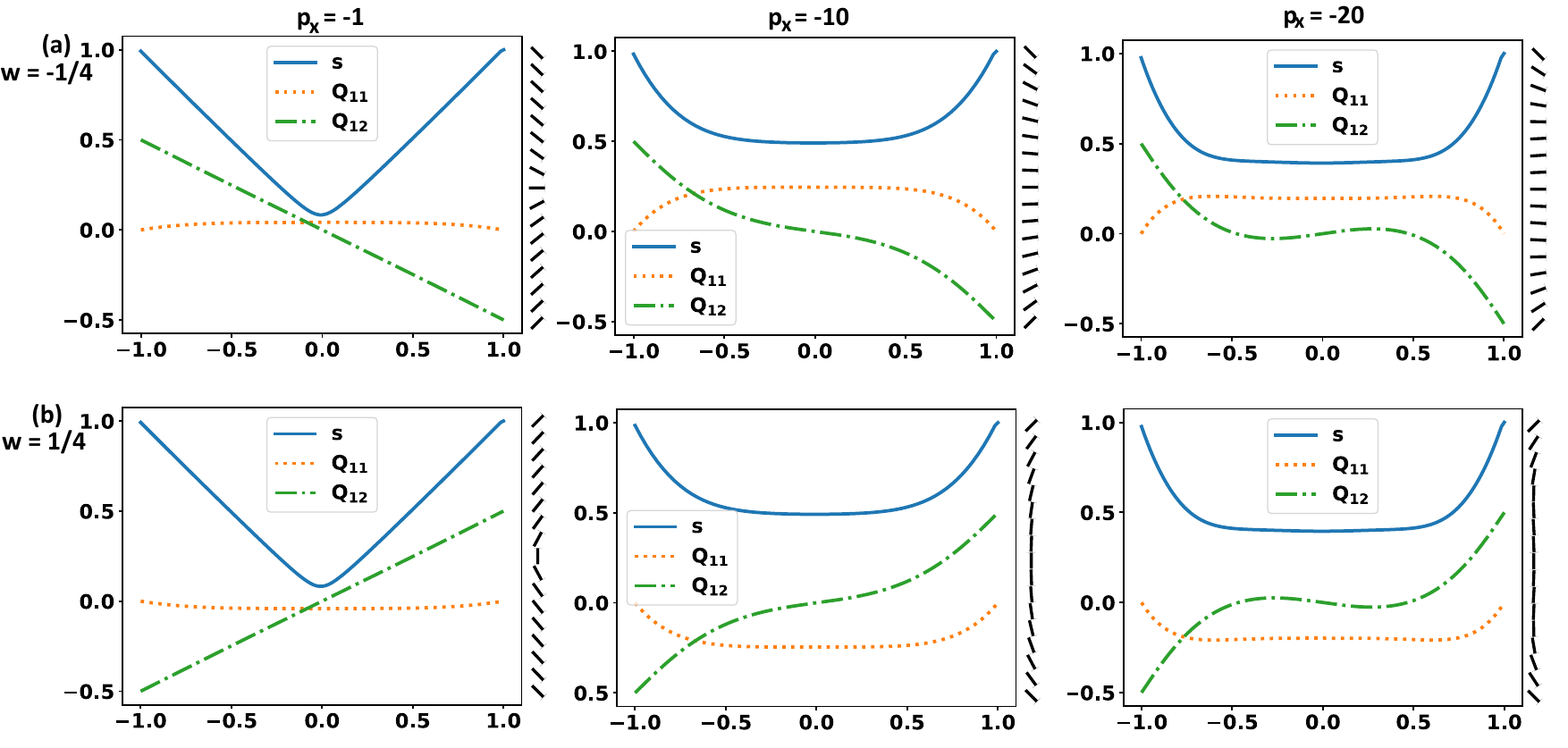}\\
    \end{minipage}
    \caption{The stable solutions of \eqref{eq:Q-flow} for  $L^*=\infty$ (i.e., we remove the bulk contributions) and $L_2=1e-3$. The values of $p_x$ and $\omega$, are indicated in the plots (the same comments apply to all other figures where values are included in the plots).}
    \label{fig:L_infinty}
\end{figure}

We now proceed to study solutions of (\ref{eq:Q-flow}) in the $L^*\to 0$ limit, relevant for micron-scale channel domains. We study the stable equilibrium solutions, the existence of OR-type solutions in this limit, and how well the OR-type solutions are approximated by the asymptotic expansions in Section~\ref{sec:passive-asymptotics}. 
As expected, in Fig. \ref{fig:L_0-stable-solutions} 
we find stable equilibria which satisfy $s=1$ almost everywhere and report unstable OR-type solutions in Fig. \ref{fig:OR-flow}, when $\omega=-\frac{1}{4}$. We again consider these to be OR-type solutions despite $s(0)\neq0$, since their behaviour is consistent with the asymptotic expressions \eqref{eq:s-comp} and \eqref{eq:theta-asymptotic}, and we also have approximate polydomain structures. We also find these OR-type solutions for $\omega=\frac{1}{4}$, but do not report them as they are similar to the $\omega=-\frac{1}{4}$ case (the same is true in the next subsection). In fact, $\omega=\pm\frac{1}{4}$ are the only boundary conditions for which we have been able to identify OR-type solutions (identical comments apply to the active case).

In Fig. \ref{fig:OR-flow}, we present three distinct OR-type solutions which vary in their $Q_{11}$ and $Q_{12}$ profiles, or equivalently the rotation of $\theta$ between the bounding plates at $y=\pm 1$. These numerical solutions are found by taking \eqref{eq:s-comp} (with $s_{min}=0$) and \eqref{eq:theta-asymptotic} with different values of $k$ ($k=0,1,2$), as the initial condition in our Newton solver. We conjecture that one could build a hierarchy of OR-type solutions corresponding to arbitrary integer values of $k$ in (\ref{eq:jump_k}), or different jumps in $\theta$ at $y=0$ in (\ref{eq:jump_k}), when $\omega = \pm \frac{1}{4}$.  
OR-type solutions are unstable, and we speculate that the solutions corresponding to different values of $k$ in (\ref{eq:jump_k}) are unstable equilibria with different Morse indices, where the Morse index is a measure of the instability of an equilibrium point \cite{han_nonlinearity}. A higher value of $k$ could correspond to a higher Morse index or informally speaking, a more unstable equilibrium point with more directions of instability. 
A further relevant observation is that according to the asymptotic expansion \eqref{eq:theta-asymptotic},  $Q_{11}(0\pm)=0$ and $Q_{12}(0\pm)=\pm \frac{1}{2}$, and hence the energy of the domain wall does not depend strongly on $k$. The far-field behavior does depend on $k$ in \eqref{eq:theta-asymptotic}, and we conjecture that this $k$-dependence generates the family of $k$-dependent OR-type solutions. 
We note that OR-type solutions generally do not satisfy $s(0) = 0$, but $s(0) \to 0$ 
as $L^*$ decreases, for a fixed $p_x$ (see Fig. \ref{fig:effect-of-decreasing-L}).  
\begin{figure}[ht]
    \centering
    \begin{minipage}{1.0\textwidth}
        \centering
        \includegraphics[width=1.0\textwidth]{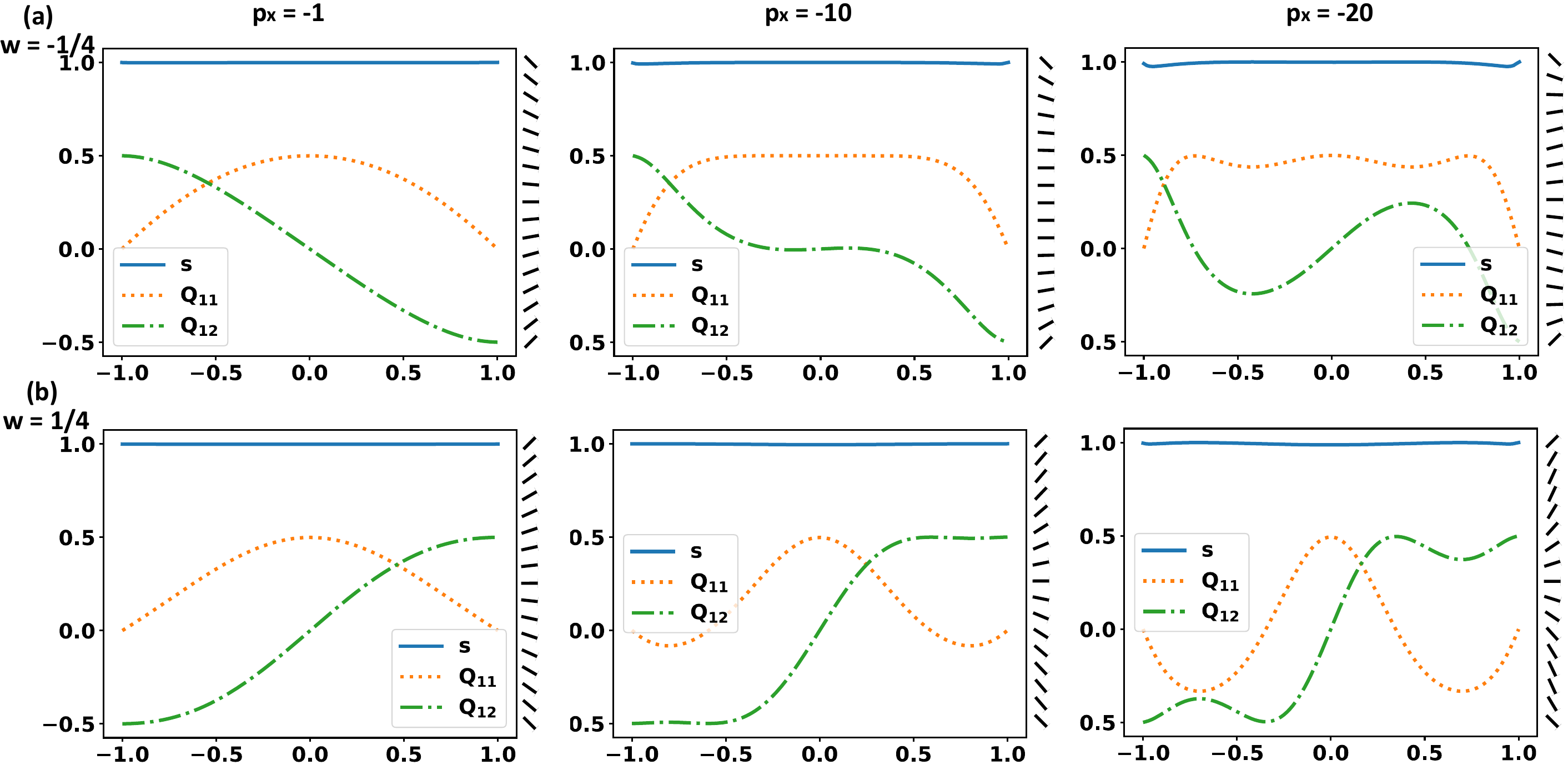}\\
    \end{minipage}
    \caption{Some example stable solutions of \eqref{eq:Q-flow} for $L^*=1e-3$ and $L_2=1e-3$. 
    }
    \label{fig:L_0-stable-solutions}
\end{figure}

To conclude this section on passive flows, we assess the accuracy of our asymptotic expansions in section~\ref{sec:passive-asymptotics}. In Fig. \ref{fig:error-plot}, we plot the error between the asymptotic expressions (\eqref{eq:s-comp} and \eqref{eq:theta-asymptotic}) and the corresponding numerical solutions of \eqref{eq:Q-flow}, for the parameter values $L^*=1e-4$, $L_2=1e-3$, $p_x=-20$ and $\omega=-\frac{1}{4}$. More precisely, we use these parameter values along with $k=1,2,3$ in \eqref{eq:theta-asymptotic}, and \eqref{eq:s-comp} with $s_{min}=0$, to construct the asymptotic profiles. We then use these asymptotic profiles as initial conditions to find the corresponding numerical solutions. Hence, we have three comparison plots in Fig. \ref{fig:error-plot}, corresponding to $k=1,2,3$ respectively.
By error, we refer to the difference between the asymptotic profile and the corresponding numerical solution. We label the asymptotic profiles using the superscript $0$, in the $L^*\to 0$ limit,  whilst a nonzero superscript identifies the numerical solution along with the value of $L^*$ used in the numerics (these comments also apply to the active case in the next section). We find good agreement between the asymptotics and numerics, especially for the $s$ profiles, where any error is confined to a narrow interval around $y=0$ and does not exceed $0.07$ in magnitude. Using \eqref{eq:Q-components}, \eqref{eq:s-comp}, and \eqref{eq:theta-asymptotic}, we construct the corresponding asymptotic profile $\Qvec^0$. Looking at the differences between $\Qvec^0$ and the numerical solutions $\Qvec^{1e-4}$ (for $k=1,2,3$), the error does not exceed $0.06$ in magnitude. This implies good agreement between the asymptotic and numerically computed $\theta$-profiles, at least for the parameter values under consideration. While the fluid velocity $u$ is not the focus of this work, we note that our asymptotic profile \eqref{eq:u-asymptotic}, gives almost perfect agreement with the numerical solution for $u$.

\begin{figure}[ht]
    \centering
    \begin{minipage}{1.0\textwidth}
        \centering
        \includegraphics[width=1.0\textwidth]{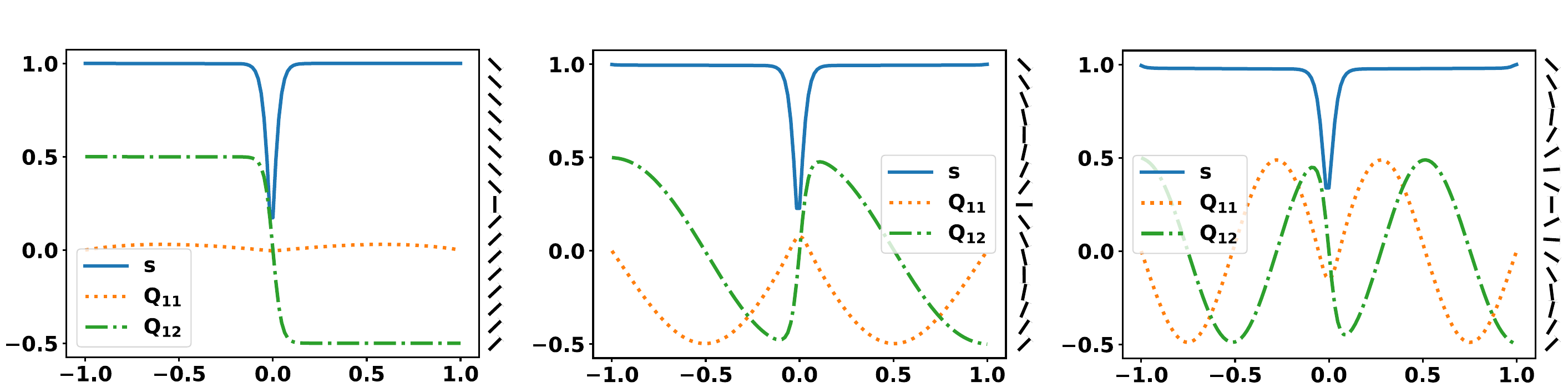}\\
    \end{minipage}
    \caption{Three unstable OR-type solutions (in the sense that they have transition layer profiles for $s$) of \eqref{eq:Q-flow} for $L^*=1e-3$, $L_2=1e-3$, $p_x=-1$ and $\omega=-\frac{1}{4}$. The initial conditions used are \eqref{eq:s-comp} (with $s_{min}=0$) and \eqref{eq:theta-asymptotic} with $k=0,1,2$ (from left to right), along with the parameter values just stated.}
    \label{fig:OR-flow}
\end{figure}

\begin{figure}[ht]
    \centering
    \begin{minipage}{1.0\textwidth}
        \centering
        \includegraphics[width=1.0\textwidth]{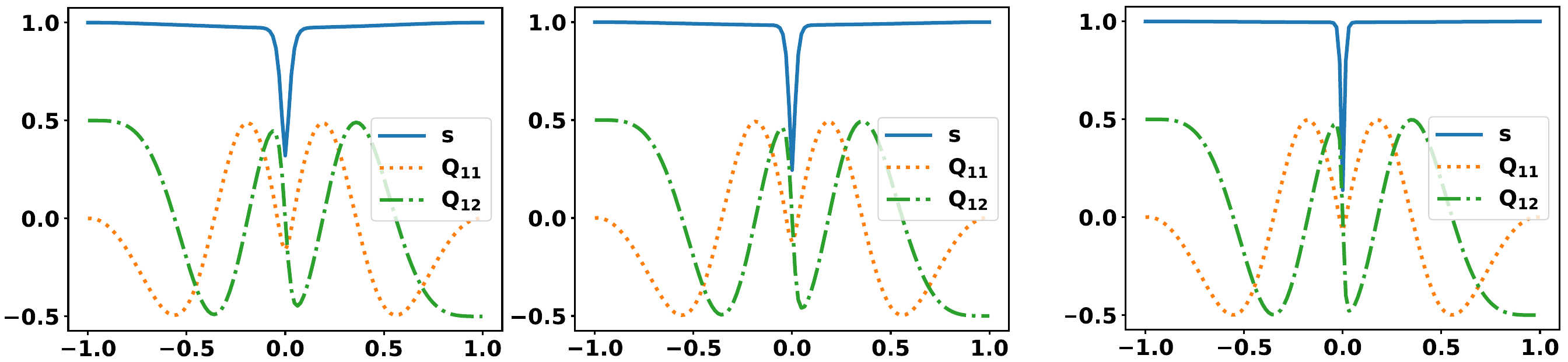}\\
    \end{minipage}
    \caption{Plot of an OR-type solution for $L^*=5e-4$, $3e-4$, $1e-4$ (from left to right). The remaining parameter values are $L_2=1e-3$, $p_x=-20$ and $\omega=-\frac{1}{4}$. The initial conditions used are \eqref{eq:s-comp} (with $s_{min}=0$) and \eqref{eq:theta-asymptotic} with $k=2$, along with the parameter values just stated. }
    \label{fig:effect-of-decreasing-L}
\end{figure}

\begin{figure}[ht]
    \centering
    \begin{minipage}{1.0\textwidth}
        \centering
        \includegraphics[width=1.0\textwidth]{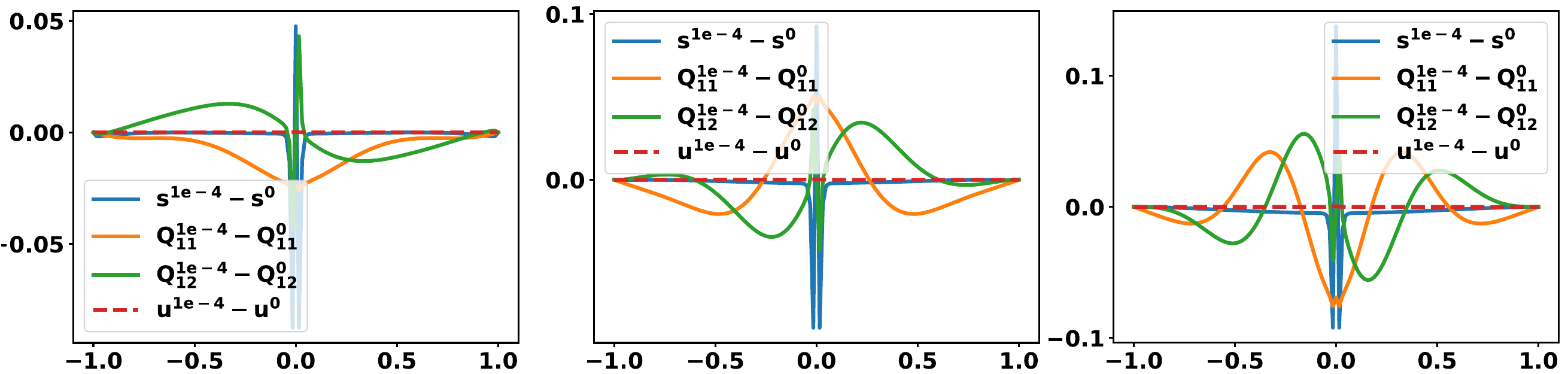}\\
    \end{minipage}
    \caption{Plot of $\Qvec^{1e-4}-\Qvec^0$, $s^{1e-4}-s^0$, and $u^{1e-4}-u^0$. Here, $\Qvec^{0}$ is the asymptotic profile given by \eqref{eq:s-comp} and \eqref{eq:theta-asymptotic} with, $s_{min}=0$, $k=1,2,3$ (from left to right), $L^*=1e-4$, $L_2=1e-3$, $p_x=-20$ and $\omega=-1/4$, whilst $\Qvec^{1e-4}$ denotes the corresponding numerical solution of \eqref{eq:Q-flow}. $s^0$ is given by \eqref{eq:s-comp} 
    and $s^{1e-4}$ is extracted from $\Qvec^{1e-4}$.
    The numerical solutions are found by using $\Qvec^0$ as the initial condition. Identical comments apply to $u^0-u^{1e-4}$, where $u^0$ is given by \eqref{eq:u-asymptotic} and $u^{1e-4}$ is the numerical solution of \eqref{eq:Q-flow}.}
    \label{fig:error-plot}
\end{figure}

\subsubsection{Active flows}

\begin{figure}[ht]
    \centering
    \begin{minipage}{1.0\textwidth}
        \centering
        \includegraphics[width=1.0\textwidth]{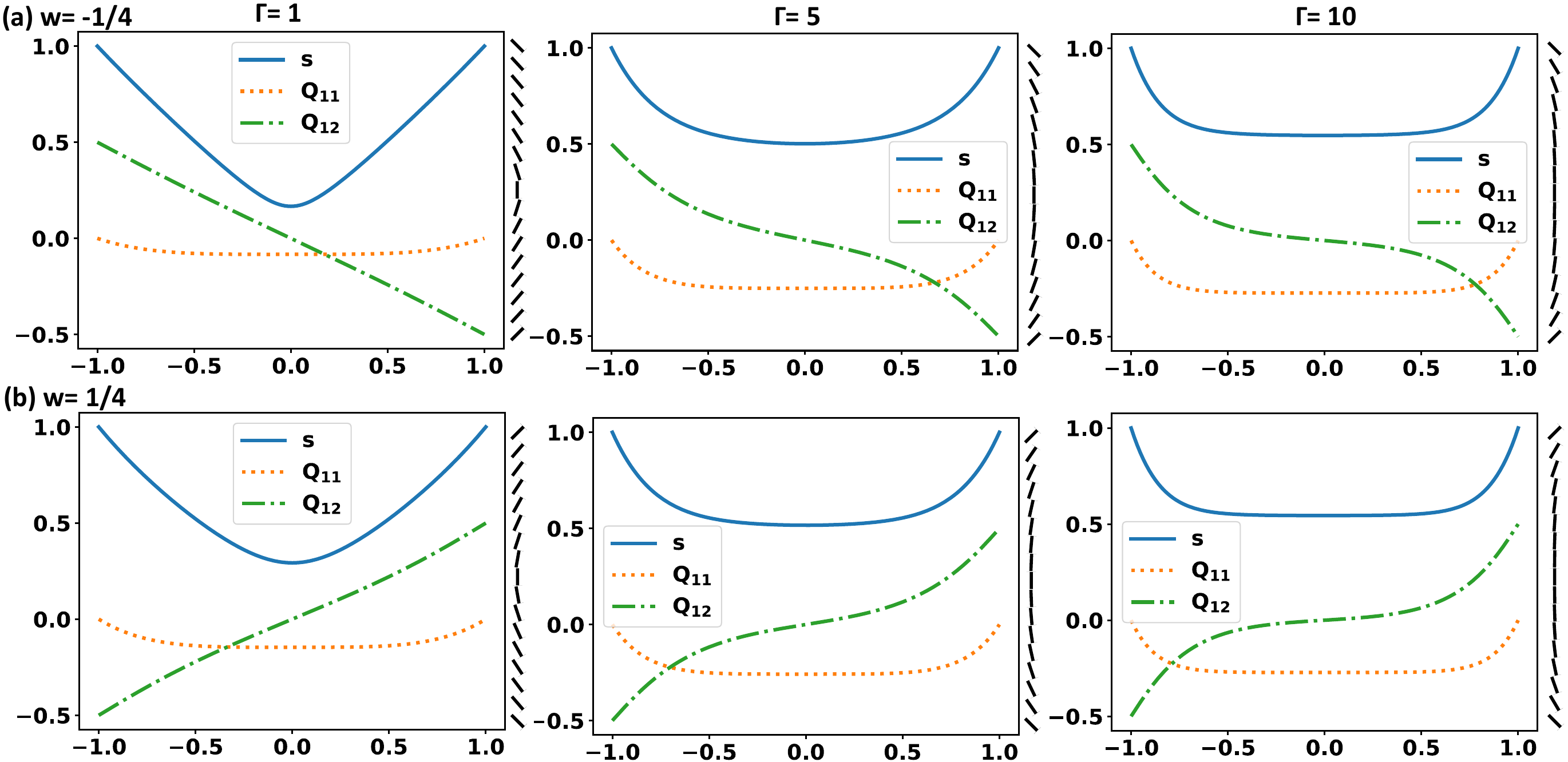}\\
    \end{minipage}
    \caption{The stable solutions of \eqref{eq:Q-active} for  $L^*=\infty$ 
    , $L_2=1e-3$, $c=\sqrt{2\pi}$ and $p_x=-1$.
    }
    \label{fig:L_infinty_active}
\end{figure}
As explained previously, we consider active flows with constant concentration $c$, and take $c>c^*$. To this end, we fix $c=\sqrt{2\pi}$ in the following numerical experiments. 
For $L^*$ large (small nano-scale channel domains), we find OR-type solutions when $\omega=\pm\frac{1}{4}$, and these are stable. 
In Fig. \ref{fig:L_infinty_active}, we plot these solutions when $p_x=-1$ and for three different values of $\Gamma$, which we recall is proportional to the activity parameter $\alpha_2$.
We only have $s(0)< 0.5 $ when $\Gamma=1$, in which case the director profile exhibits polydomain structures. 
As $\Gamma$ increases, $s(0)$ increases and $s \to 1$ almost everywhere, so that OR-type solutions are only possible for small values of $p_x$ and $\Gamma$. 
Increasing $|p_x|$ for a fixed value of $\Gamma$, also drives $s\to1$ everywhere. 
\begin{figure}[t]
    \centering
    \begin{minipage}{1.0\textwidth}
        \centering
        \includegraphics[width=1.0\textwidth]{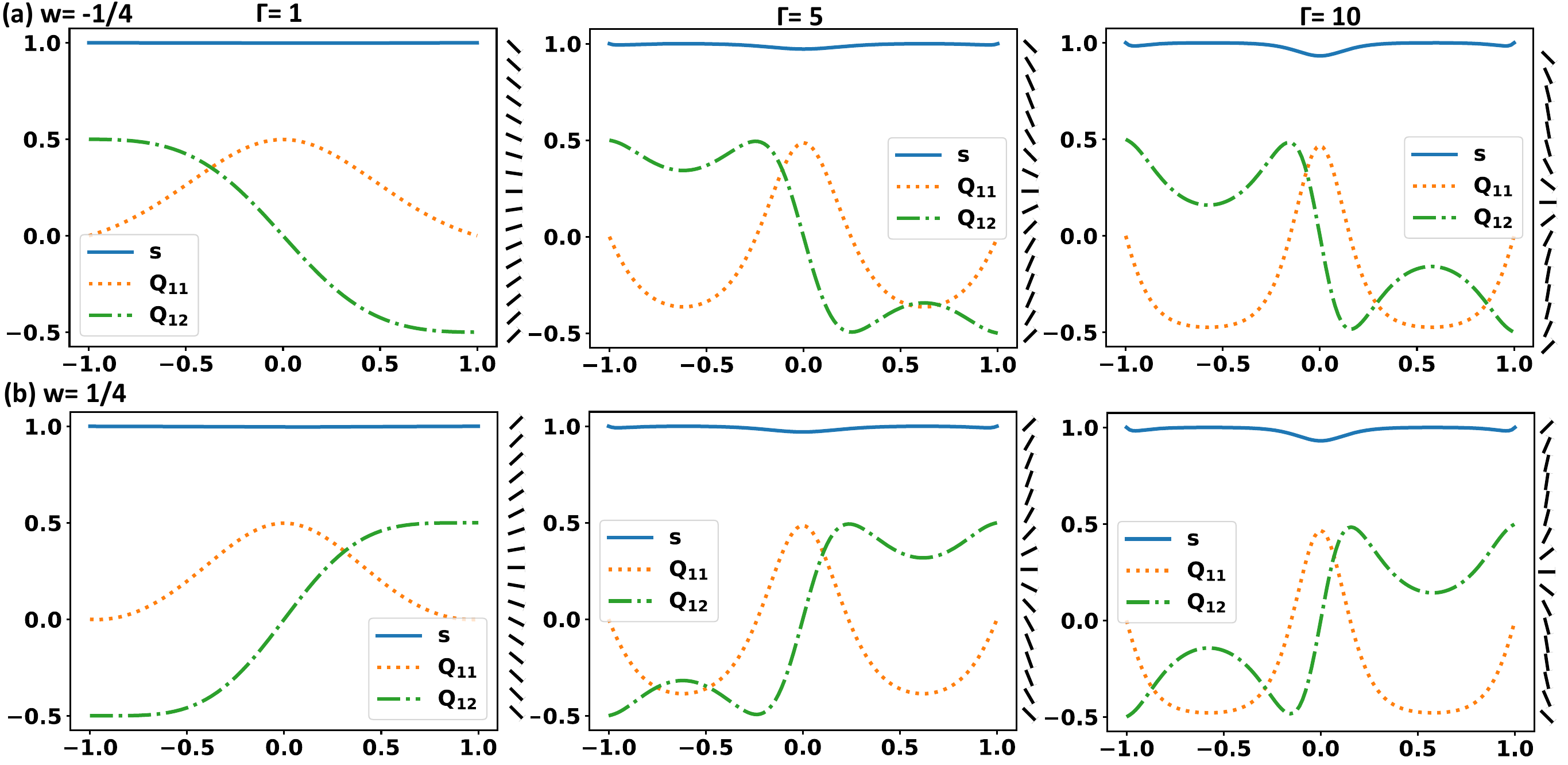}\\
    \end{minipage}
    \caption{The stable solutions of \eqref{eq:Q-active} for $L^*=1e-3$, $L_2=1e-3$, $c=\sqrt{2\pi}$ 
    and $p_x=-1$.}
    \label{fig:L_small_stable_solutions}
\end{figure}

As in the passive case, we also find unstable OR-type solutions consistent with the limiting asymptotic expression \eqref{eq:s-comp}, 
for small values of $L^*$ that correspond to micron-scale channels. The stable solutions have $s \approx 1$ almost everywhere (see Fig. \ref{fig:L_small_stable_solutions}). In Fig. \ref{fig:active-OR}, we find unstable OR-type solutions when $L^*=1e-3$, $L_2=1e-3$ and $\omega=-\frac{1}{4}$, for a range of values of $p_x$ and $\Gamma$. To numerically compute these solutions, we use the stated parameter values in \eqref{eq:s-comp} (with $s_{min}=0$) and \eqref{eq:theta-active-explicit}, along with $k=0$, as our initial condition. We only have $s(0)\approx 0$ 
provided $|p_x|$ and $\Gamma$ are not too large, however, $s(0) \to 0$ in the $L^*\to 0$ limit for fixed values of $p_x$ and $\Gamma$. This illustrates the robustness of OR-type solutions in an active setting. In Fig. \ref{fig:active-k-plot}, we plot three further distinct OR-type solutions, obtained by taking \eqref{eq:s-comp} (with $s_{min}=0$) and \eqref{eq:theta-active-explicit} with $k=1,2,3$, as our initial condition.
Hence, for the same reasons as in the passive case, we believe there may be multiple unstable OR-type solutions, corresponding to different values of $k$ in (\ref{eq:jump_k}).
\begin{figure}[ht]
    \centering
    \begin{minipage}{1.0\textwidth}
        \centering
        \includegraphics[width=1.0\textwidth]{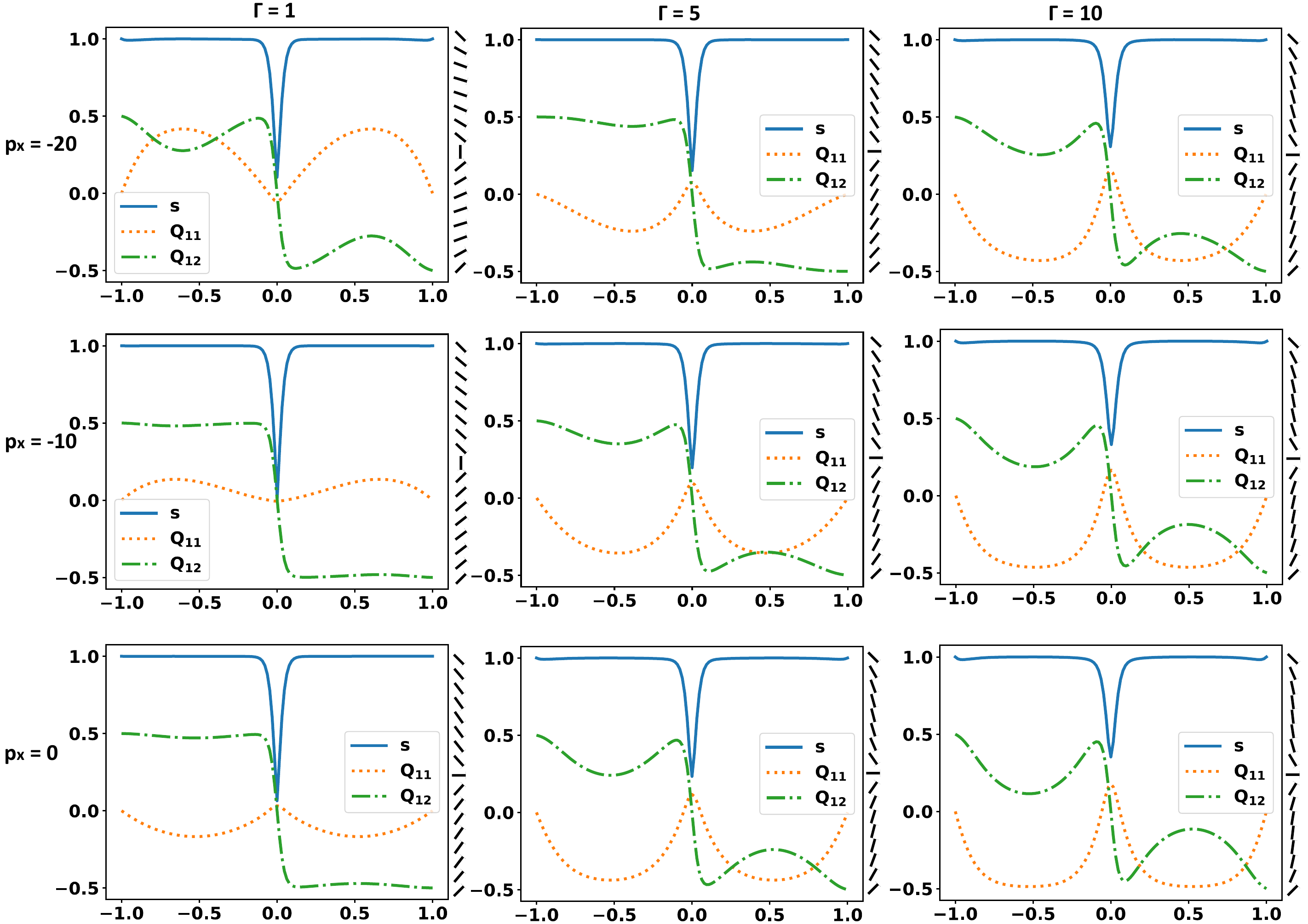}\\
    \end{minipage}
    \caption{Unstable OR-type solutions (in the sense that they have transition layer profiles for $s$) of \eqref{eq:Q-active}, for $L^*=1e-3$, $L_2=1e-3$, $c=\sqrt{2\pi}$ and $\omega=-\frac{1}{4}$. 
    The initial conditions used are \eqref{eq:s-comp} (with $s_{min}=0$) and \eqref{eq:theta-active-explicit} with $k=0$.}
    \label{fig:active-OR}
\end{figure}

\begin{figure}[ht]
    \centering
    \begin{minipage}{1.0\textwidth}
        \centering
        \includegraphics[width=1.0\textwidth]{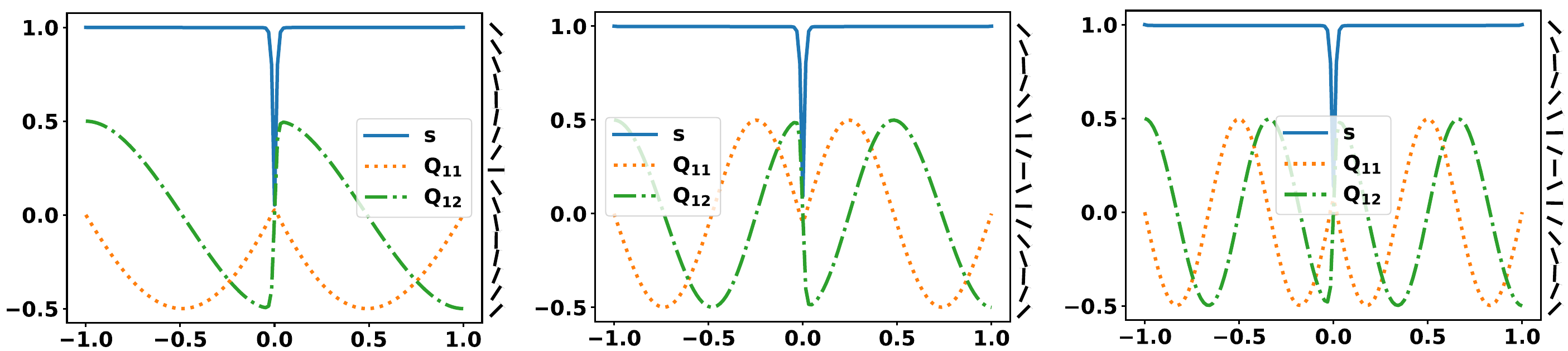}\\
    \end{minipage}
    \caption{Three unstable OR-type solutions of \eqref{eq:Q-active} for $L^*=1e-3$, $L_2=1e-3$, $p_x=-1$, $\Gamma=0.7$ and $\omega=-\frac{1}{4}$.} 
    \label{fig:active-k-plot}
\end{figure}

By analogy with the passive case, we now compare the asymptotic expressions \eqref{eq:s-comp}, \eqref{eq:u-active-asymptotic} and \eqref{eq:theta-active-explicit}, with the numerical solutions. 
The error plots are given in Fig. \ref{fig:active-error-plot}. Once again, there is good agreement between the limiting $s$-profile \eqref{eq:s-comp} and the numerical solutions, where any error is confined to a small interval around $y=0$. There is also good agreement between the asymptotic and numerically computed $\theta$-profiles (coded in terms of $Q_{11}$ and $Q_{12}$) and flow profile $u$, provided $|p_x|$, $\Gamma$, or both, are not too large. 
When $|p_x|$ and $\Gamma$ are large (say much greater than $1$), the accuracy of the asymptotics breaks down, especially for the $u$-profile. 
However, OR-type solutions are still possible for large values of $|p_x|$ and $\Gamma$, as elucidated by Fig. \ref{fig:active-OR}.   
\begin{figure}[ht]
    \centering
    \begin{minipage}{1.0\textwidth}
        \centering
        \includegraphics[width=1.0\textwidth]{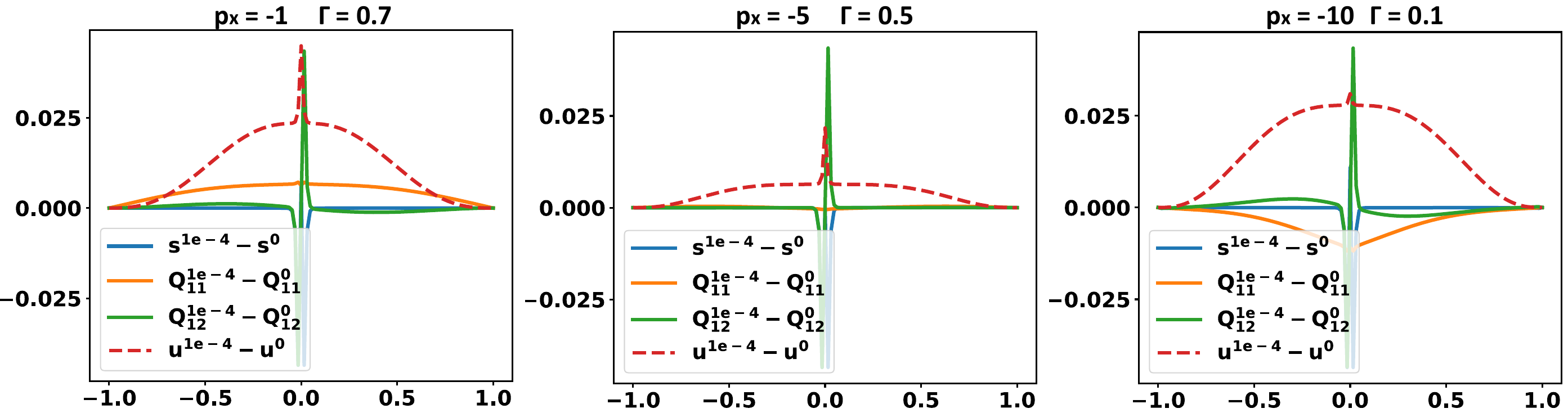}\\
    \end{minipage}
    \caption{Plot of $\Qvec^{1e-4}-\Qvec^0$, $s^{1e-4}-s^0$, and $u^{1e-4}-u^0$. Here, $\Qvec^{0}$ is given by \eqref{eq:s-comp} and \eqref{eq:theta-active-explicit} with, $s_{min}=0$, $k=0$, $c=\sqrt{2\pi}$, $L^*=1e-4$, $L_2=1e-3$, $p_x$ and $\Gamma$ as stated in the figure, and $\omega=-1/4$, whilst $\Qvec^{1e-4}$ is the numerical solution of \eqref{eq:Q-active}, with the same parameter values.} 
    \label{fig:active-error-plot}
\end{figure}

\section{Conclusions}\label{sec:conclusions}
In this article, we have demonstrated the universality of OR-type solutions in NLC-filled microfluidic channels. Section \ref{sec:constant-pressure-flow} focuses on the simple and idealised case of constant flow and pressure to give some preliminary insight into the more complex systems considered in section \ref{sec:passive+active}. 
We employ an $(s, \theta)$-formalism for the NLC state, and impose Dirichlet conditions for $(s, \theta)$ coded in terms of $\omega$, where $\omega$ is a measure of the director rotation between the bounding plates $y=\pm 1$. We always have a unique smooth solution in this framework, provided an OR solution does not exist 
(Theorem \ref{thm:uniqueness}). Additionally, in the $\Qvec$-framework, we prove OR solutions are compatible with $\omega = \pm \frac{1}{4}$ only (\Cref{thm:uniqueness+OR}), i.e., when the boundary conditions are orthogonal to each other. These OR solutions with polydomain structures exist for all values of $L^*$ or $\epsilon$, they are globally stable for large $L^*$ (small $\epsilon$), and there are multiple solutions for small values of $L^*$ (large $\eps$) or large channel geometries. In fact, for all three scenarios considered in this paper, we have found OR and OR-type solutions to be compatible with $\omega=\pm\frac{1}{4}$ only, or orthogonal boundary conditions.  
As has been noted in \cite{gartland} amongst others, orthogonal boundary conditions allow for solutions in the $\Qvec$-formalism (solutions of \eqref{eq:Q-EL-eqnts}) that have a constant set of eigenvectors in space. These 
solutions with a constant set of eigenvectors, are precisely the OR solutions which are disallowed for non-orthogonal boundary conditions. Thus, whilst the conclusion of Theorem~\ref{thm:uniqueness+OR} is not surprising, we now provide a proof of this fact. 

In section \ref{sec:passive+active}, we calculate useful asymptotic expansions for OR-type solutions in the limit of large domains, 
for both passive and active nematics. The asymptotics are validated by numerically-computed OR-type solutions for small and large values of $L^*$, using the asymptotic expansions as initial conditions. There is good agreement between the asymptotics and the numerical solutions, and the asymptotics give informative insight into the internal structure of domain walls of OR-type solutions and the outer far-field solutions. These techniques can be embellished to include external fields, other types of boundary conditions, and more complex geometries as well. 

In section \ref{sec:numerics}, the OR-type solutions are unstable for small $L^*$ or large channels. However, they may still be observable and hence, physically relevant. For example, in the experimental results in \cite{agha} for passive NLC-filled microfluidic channels, the authors find disclination lines at the centre of a microfluidic channel filled with the liquid crystal 5CB, with flow, both with and without an applied electric field. Moreover, the authors are able to stabilise these disinclination lines by applying an electric field. 
In the active case, there are similar experimental results in \cite{guillamat}. 
Here the authors apply a magnetic field to 8CB in the smectic-A phase placed on top of an aqueous gel of microtubules cross-linked by ATP-activated kinesin motor  clusters (constituting the active nematic system), and observe the formation of parallel lanes of defect cores in the active nematic, aligned perpendicularly to the magnetic field. 
These defect cores and disclination lines can be modelled by OR-type solutions, as studied in this paper, and we argue that whilst OR-type solutions are unstable for large domains, they can still influence non-equilibrium properties or perhaps be stabilised for tailor-made applications (also see \cite{han_nonlinearity}).

To conclude, we argue why OR-type solutions maybe universal in variational theories, with free energies that employ a Dirichlet elastic energy for the unknowns, e.g. $y_1 \ldots y_n$ for $n \in \mathbb{N}$. 
Working in one-dimensions, consider an energy of the form 
\begin{equation}
    \int_\Omega y^\prime_{1}(x)^2 
    +\ldots y^\prime_{n}(x)^2
    +\frac{1}{L^*} h(y_1,\ldots y_{n})(x)~\mathrm{d}x,
\end{equation}
subject to Dirichlet boundary conditions, for a material-dependent positive elastic constant $L^*$. The function, $h$, models a bulk energy that only depends on $y_1, \ldots, y_n$. As $L^*\to \infty$, the limiting Euler-Lagrange equations 
admit unique solutions of the form $y_j = ax + b$, for constants $a$ and $b$. 
For specific choices of $\Omega$ and asymmetric boundary conditions, we can have domain walls at $x=x^*$ such that $y_j(x^*) = 0$ for $j=1,\ldots, n$. Writing each $y_j = |y_j| sgn (y_j)$, the domain wall separates polydomains with phases differentiated by different values of $sgn(y_j)$. Moreover, we believe this argument can be extended to systems in two and three-dimensions.


\section*{Acknowledgments}
We thank Giacomo Canevari for helpful comments on some of the proofs in Section 3. 

\section*{Taxonomy}
The author names are listed alphabetically. JD led the project, which was conceived and designed by AM and LM. YH produced all the numerics and contributed to the analysis.
JD, AM and LM wrote the manuscript carefully and oversaw the project evolution. AM mentored JD and YH throughout the project.

\bibliographystyle{siamplain}
\bibliography{main}

\begin{thebibliography}{10}

\bibitem{agha}
{\sc H.~Agha and C.~Bahr}, {\em Nematic line defects in microfluidic channels:
  wedge, twist and zigzag disclinations}, Soft Matter, 14 (2018), pp.~653--664.

\bibitem{berisreference}
{\sc A.~N. Beris and B.~J. Edwards}, {\em Thermodynamics of Flowing Systems:
  With Internal Microstructure}, Oxford University Press, Oxford, UK, 1994.

\bibitem{brezis_bethuel_helein1993}
{\sc F.~Bethuel, H.~Brezis, and F.~H\'elein}, {\em {Asymptotics for the
  minimization of a Ginzburg--Landau functional}}, Calc. Var. Partial Diff., 1
  (1993), pp.~123--148.

\bibitem{braides}
{\sc A.~Braides}, {\em A handbook of {$\Gamma$}-convergence}, in Handbook of
  Differential Equations: Stationary Partial Differential Equations, vol.~3,
  Elsevier, North-Holland, Amsterdam, 2006, pp.~101--213.

\bibitem{CM}
{\sc M.~Calderer and B.~Mukherjee}, {\em Chevron patterns in liquid crystal
  flows}, Physica D: Nonlinear Phenomena, 98 (1996), p.~201–224.

\bibitem{Harris}
{\sc G.~Canevari, J.~Harris, A.~Majumdar, and Y.~Wang}, {\em {The well order
  reconstruction solution for three-dimensional wells, in the Landau--de Gennes
  theory}}, Int. J. Non-Linear Mech., 119 (2020), p.~103342.

\bibitem{canevari}
{\sc G.~Canevari, A.~Majumdar, and A.~Spicer}, {\em {Order reconstruction for
  nematics on squares and hexagons: a Landau--de Gennes study}}, SIAM J. Appl.
  Math., 77 (2019), pp.~267--293.

\bibitem{copar}
{\sc S.~{\v{C}}opar, {\v{Z}}.~Kos, T.~Emer{\v{s}}i{\v{c}}, and U.~Tkalec}, {\em
  Microfluidic control over topological states in channel-confined nematic
  flows}, Nat. Commun., 11 (2020), pp.~1--10.

\bibitem{optofluidic}
{\sc J.~Cuennet, A.~E. Vasdekis, and D.~Psaltis}, {\em Optofluidic-tunable
  color filters and spectroscopy based on liquid-crystal microflows}, Lab on a
  Chip, 13 (2013), pp.~2721--2726.

\bibitem{dalby-farrell-majumdar-xia}
{\sc J.~Dalby, P.~Farrell, A.~Majumdar, and J.~Xia}, {\em {One-Dimensional
  Ferronematics in a Channel: Order Reconstruction, Bifurcations and
  Multistability}}, SIAM J. on Appl. Math., 82 (2022), pp.~694--719.

\bibitem{deGennes}
{\sc P.~G. de~Gennes}, {\em {The Physics of Liquid Crystals}}, Oxford
  University Press, Oxford, 1974.

\bibitem{review}
{\sc A.~Doostmohammadi, J.~Ign\'es-Mullol, J.~M. Yeomans, and F.~Sagu\'es},
  {\em {Active nematics}}, Nat. Commun., 9 (2018), p.~3246.

\bibitem{edwards_active}
{\sc S.~A. Edwards and J.~M. Yeomans}, {\em {Spontaneous flow states in active
  nematics: A unified picture}}, Europhysics letters, 85 (2005), p.~18008.

\bibitem{fang2019}
{\sc L.~Fang, A.~Majumdar, and L.~Zhang}, {\em Surface, size and topological
  effects for some nematic equilibria on rectangular domains}, Math. Mech.
  Solids, 25 (2020), pp.~1101--1123.

\bibitem{Giomo_annihilation}
{\sc L.~Giomi, M.~Bowick, X.~Ma, and M.~Marchetti}, {\em Defect annihilation
  and proliferation in active nematics}, Phys. Rev. Lett., 110 (2013),
  pp.~228101--1--228101--5.

\bibitem{active-defects}
{\sc L.~Giomi, M.~Bowick, P.~Mishra, R.~Sknepnek, and M.~Marchetti}, {\em
  {Defect dynamics in active nematics}}, Phil. Trans. R.Soc. A, 372 (2014),
  p.~20130365.

\bibitem{giomi_sheared}
{\sc L.~Giomi, T.~B. Liverpool, and M.~Marchetti}, {\em {Sheared active fluids:
  Thickening, thinning, and vanishing viscosity}}, Phys. Rev. E, 81 (2010),
  p.~051908.

\bibitem{Giomo-2012}
{\sc L.~Giomi, L.~Mahadevan, B.~Chakraborty, and M.~Hagan}, {\em {Banding,
  excitability and chaos in active nematic suspensions}}, Nonlinearity, 25
  (2012), p.~2245–2269.

\bibitem{excitable}
{\sc L.~Giomi, L.~Mahadevan, B.~Chakraborty, and M.~F. Hagan}, {\em Excitable
  patterns in active nematics}, Phys. Rev. L., 106 (2011), p.~218101.

\bibitem{golovaty2015}
{\sc D.~Golovaty, J.~Montero, and P.~Sternberg}, {\em {Dimension Reduction for
  the Landau-de Gennes Model in Planar Nematic Thin Films}}, J. Nonlinear Sci.,
  25 (2015), pp.~1431--1451.

\bibitem{guillamat}
{\sc P.~Guillamat, J.~Ign{\'e}s-Mullol, and F.~Sagu{\'e}s}, {\em Control of
  active liquid crystals with a magnetic field}, Proc. Natl. Acad. Scis, 113
  (2016), pp.~5498--5502.

\bibitem{han2020siap}
{\sc Y.~Han, A.~Majumdar, and L.~Zhang}, {\em A reduced study for nematic
  equilibria on two-dimensional polygons}, SIAM J. Appl. Math., 80 (2020),
  pp.~1678--1703.

\bibitem{han_nonlinearity}
{\sc Y.~Han, J.~Yin, P.~Zhang, A.~Majumdar, and L.~Zhang}, {\em Solution
  landscape of a reduced landau-de gennes model on a hexagon}, Nonlinearity, 34
  (2021), pp.~2048--2069.

\bibitem{lamy}
{\sc X.~Lamy}, {\em Bifurcation analysis in a frustrated nematic cell}, J.
  Nonlinear Sci., 24 (2014), pp.~1197--1230.

\bibitem{lewis2014}
{\sc A.~H. Lewis, I.~Garlea, J.~Alvarado, O.~J. Dammone, P.~D. Howell,
  A.~Majumdar, B.~M. Mulder, M.~Lettinga, G.~H. Koenderink, and D.~G. Aarts},
  {\em {Colloidal liquid crystals in rectangular confinement: Theory and
  experiment}}, Soft Matter, 10 (2014), p.~7865–7873.

\bibitem{logg2012automated}
{\sc A.~Logg, K.-A. Mardal, and G.~Wells}, {\em Automated solution of
  differential equations by the finite element method: The FEniCS book},
  vol.~84, Springer Science \& Business Media, 2012.

\bibitem{majumdar-2010-article}
{\sc A.~Majumdar}, {\em {Equilibrium order parameters of nematic liquid
  crystals in the Landau--de Gennes theory}}, Euro. J. Appl. Math, 21 (2010),
  pp.~181--203.

\bibitem{marchetti}
{\sc M.~C. Marchetti, J.-F. Joanny, S.~Ramaswamy, T.~B. Liverpool, J.~Prost,
  M.~Rao, and R.~A. Simha}, {\em Hydrodynamics of soft active matter}, Rev.
  Mod. Phys., 85 (2013), p.~1143.

\bibitem{mondal}
{\sc S.~Mondal, I.~Griffiths, F.~Charlet, and A.~Majumdar}, {\em Flow and
  nematic director profiles in a microfluidic channel: the interplay of nematic
  material constants and backflow}, Fluids, 3 (2018), p.~39.

\bibitem{MottramBiaxialChevron}
{\sc N.~Mottram, N.~U. Islam, and S.~Elston}, {\em Biaxial modeling of the
  structure of the chevron interface in smectic liquid crystals}, Phys. Rev. E,
  60 (1999), pp.~613--619.

\bibitem{PhillipsChevron}
{\sc L.~Mrad and D.~Phillips}, {\em Dynamic analysis of chevron structures in
  liquid crystal cells}, Mol. Cryst. Liq. Cryst., 647 (2017), pp.~66--91.

\bibitem{gartland}
{\sc P.~Palffy-muhoray, E.~C. Gartland~Jr, and J.~R. Kelly}, {\em {A new
  configurational transition in inhomogeneous nematics}}, Liquid Crystals, 16
  (1994), pp.~713--718.

\bibitem{RiekerChevron}
{\sc T.~P. Rieker, N.~A. Clark, G.~S. Smith, D.~S. Parmar, E.~B. Sirota, and
  C.~R. Safinya}, {\em ``chevron" local layer structure in surface-stabilized
  ferroelectric smectic-$c$ cells}, Phys. Rev. Lett., 59 (1987),
  pp.~2658--2661.

\bibitem{sluckin}
{\sc N.~Schopohl and T.~J. Sluckin}, {\em {Defect Core Structure in Nematic
  Liquid Crystals}}, Phys. Rev. Lett., 59 (1987), pp.~2582--2584.

\bibitem{microcargo}
{\sc A.~Sengupta, C.~Bahr, and S.~Herminghaus}, {\em Topological microfluidics
  for flexible micro-cargo concepts}, Soft Matter, 9 (2013), pp.~7251--7260.

\bibitem{tsakonas}
{\sc C.~Tsakonas, A.~J. Davidson, C.~V. Brown, and N.~J. Mottram}, {\em
  Multistable alignment states in nematic liquid crystal filled wells}, Applied
  physics letters, 90 (2007), p.~111913.

\bibitem{wang_canevari_majumdar_2019}
{\sc Y.~Wang, G.~Canevari, and A.~Majumdar}, {\em {Order reconstruction for
  nematics on squares with isotropic inclusions: a Landau--de Gennes study}},
  SIAM J. Appl. Math., 79 (2019), pp.~1314--1340.

\end{thebibliography}

\newpage
\setcounter{section}{0}

\section*{Supplementary materials. A Multi-Faceted Study of Nematic Order Reconstruction in Microfluidic Channels}

\section{Supplementary material for section \ref{sec:theory} - Theory} 
Here we give further details of the reduced modelling approach captured by \eqref{eq:Q}.

The reduced $\Qvec$-tensor \eqref{eq:Q}, is reasonable from a modelling perspective in certain
physical settings such as ours. Recall, we consider a thin channel so that we assume structural properties are invariant in the $z$-direction and we can consider a two-dimensional domain in the $xy$-plane. The reduction from a three-dimensional domain to a two-dimensional problem for thin film systems is reasonable on experimental grounds, but can also be justified rigorously. In \cite{golovaty2015} (also see \cite{wang_canevari_majumdar_2019} Theorem 2.1), the authors consider a three-dimensional thin film of nematic liquid crystal, on which they impose planar surface anchoring conditions on the top and bottom surfaces of the film, along with uniaxial $z$-invariant Dirichlet conditions on the lateral surfaces. In Theorem 5.1 of \cite{golovaty2015}, the authors use techniques from $\Gamma$-convergence to prove that when the height of the film is sufficiently small or in the thin film limit, it suffices to study the modelling problem (or the LdG energy minimization problem) on the planar cross-section with a two-dimensional domain. Our two-dimensional domain is $D = \left\{(x,y): -D \leq x \leq D; -L \leq y \leq L \right\}$, and $L << D$ by assumption. On these grounds, we further assume that the structural details are invariant in the $x$-direction and it suffices to work with a one-dimensional channel, $y \in [-L, L]$.

A further consequence of this result is the emergence of the reduced $\mathbf{Q}$-tensor in \eqref{eq:Q}. In \cite{golovaty2015}, the authors consider a full Landau-de Gennes (LdG) $\Qvec_f$-tensor, i.e., a $3\times 3$ symmetric traceless matrix
\begin{equation}\label{eq:full_Q}
            \Qvec_f=\begin{pmatrix}
            Q_{11} & Q_{12} & Q_{13}\\
            Q_{12} & Q_{22} & Q_{23}\\
            Q_{13} & Q_{23} & -Q_{11}-Q_{22}
        \end{pmatrix}.
\end{equation}
\noindent The authors impose a surface energy on the top and bottom of the film, which induces planar degenerate boundary conditions or enforces planar alignment of the corresponding nematic molecules on these surfaces. In other words, the minimizer, $\Qvec_s$, of the imposed surface energy on the top and bottom surfaces has the leading eigenvector (with the largest positive eigenvalue) in the $xy$-plane, with a fixed eigenvector in the $z$-direction and a constant eigenvalue associated with the fixed eigenvector in the $z$-direction (at least for a range of choices of the parameters in the surface energy, which comply with our modelling set-up).
In the thin film limit, the authors prove a $\Gamma$-convergence result and the minimizers of the $\Gamma$-limit of the LdG energy belong to the space of minimizers of the imposed surface energy on the top and bottom surfaces i.e. the candidate physically relevant configurations or minimizers of the LdG energy, labelled as $\Qvec_f$, have a fixed eigenvector in the $z$-direction with a fixed computable eigenvalue (in terms of the parameters in the surface energy). In our context, this means $\mathbf{e}_z$ (the unit-vector in the $z$-direction) is a fixed eigenvector for the physically relevant $\Qvec_f$. A simple calculation shows that $Q_{13}=Q_{23}=0$ and we relabel the components of \eqref{eq:full_Q} as follows,
\begin{align}
            \label{eq:Q_3d}
            \Qvec_f&=
            \begin{pmatrix}
            Q_{11}-q_3 & Q_{12} & 0\\
            Q_{12} & -Q_{11}-q_3 & 0\\
            0 & 0 & 2q_3
        \end{pmatrix}\nonumber\\
        &=
        \begin{pNiceMatrix}
        \Block{2-2}<\large>{\Qvec}  & & 0 \\
        & &  0 \\
        0 & 0 & 0
        \end{pNiceMatrix}
        + (2\mathbf{e}_z \otimes\mathbf{e}_z-\mathbf{e}_x \otimes\mathbf{e}_x-\mathbf{e}_y \otimes\mathbf{e}_y)q_3\nonumber\\
        &=Q_{11}(\mathbf{e}_x \otimes\mathbf{e}_x-\mathbf{e}_y \otimes\mathbf{e}_y) + Q_{12}(\mathbf{e}_x \otimes\mathbf{e}_y + \mathbf{e}_y \otimes\mathbf{e}_x)\nonumber\\
        &\quad\quad\quad\quad+ (2\mathbf{e}_z \otimes\mathbf{e}_z-\mathbf{e}_x \otimes\mathbf{e}_x-\mathbf{e}_y \otimes\mathbf{e}_y)q_3,\nonumber
\end{align}
where $\Qvec$ is given by \eqref{eq:Q}, and $q_3$ is a fixed known constant (at least for certain physically relevant situations (as captured by \cite{golovaty2015} Theorem 3.1 case (iv)) and/or system temperatures (\cite{canevari})).
This leaves us with two degrees freedom and this is precisely the information captured by the reduced $\Qvec$-tensor \eqref{eq:Q}. Hence, \eqref{eq:Q} describes the nematic ordering in the $xy$-plane.

Finally, we note that there is no notion of uniaxiality or biaxiality associated to the reduced $\Qvec$-tensor \eqref{eq:Q}. When mapped to the full $\Qvec_f$-tensor as described above, we can address these questions via the biaxiality parameter $\beta:=1-(\textrm{tr}\Qvec_f^3)^2/(\textrm{tr}\Qvec_f^2)^3$. As in previous works on OR (e.g. the seminal work \cite{sluckin}), when placed in the context of the full $\Qvec_f$-tensor, our OR solutions in the main text (\Cref{sec:constant-pressure-flow}) reconcile our conflicting uniaxial boundary conditions via the introduction of biaxiality. We now demonstrate this with an example (also see \cite{han2020siap}).

Consider a full $\Qvec_f$-tensor subject to the following uniaxial Dirichlet boundary conditions:
\begin{equation}\label{eq:3D_bcs}
    \Qvec_\textbf{b}(\mathbf{x})=\begin{cases}
        s_+\left(\nvec_1\otimes\nvec_1-\frac{1}{3}\mathbf{I}\right) \quad\textrm{on $y=-1$}\\
        s_+\left(\nvec_2\otimes\nvec_2-\frac{1}{3}\mathbf{I}\right) \quad\textrm{on $y=1$},
    \end{cases}
\end{equation}
where 
\begin{equation}
    \nvec_1=\frac{1}{\sqrt{2}}(1,-1,0) \textrm{ and } \nvec_2=\frac{1}{\sqrt{2}}(1,1,0),
\end{equation}
i.e., the boundary conditions \eqref{eq:Q-bcs} with $\omega=\frac{1}{4}$ and the third dimension  included. Here,
\begin{equation}
    s_+=\frac{B+\sqrt{B^2-24AC}}{4C},
\end{equation}
is the value of the scalar order parameter $s$, such that bulk potential
\begin{equation}
    f_b(\Qvec)=\frac{A}{2}\textrm{tr}(\Qvec^2) -\frac{B}{3}\textrm{tr}(\Qvec^3) +\frac{C}{4}(\textrm{tr}\Qvec^2)^2,
\end{equation}
is minimised in the set $\mathcal{N}=\{\Qvec\in S_3:\Qvec=s_+(\nvec\otimes\nvec-\mathbf{I}/3)\}$ \cite{majumdar-2010-article} (note,  $S_3\coloneqq\{\Qvec\in\mathbb{M}^{3\times 3}: Q_{ij}=Q_{ji},Q_{ii}=0\}$). As in the main text, $A<0$ and $C>0$ are temperature and material dependent constants respectively, while $B>0$ is another material dependent constant.
We consider the special temperature $A=-B^2/3C$ and $q_3=-B/6C$ \cite{canevari}.
Letting $Q_{11}$ and $Q_{12}$ denote an OR solution of \eqref{eq:Q11eqtn}-\eqref{eq:Q12eqtn}, subject to \eqref{eq:Q-bcs} with $\omega=\frac{1}{4}$,
\begin{equation}\label{eq:mapping}
    \Qvec_f=
            \begin{pmatrix}
            s_+Q_{11}-q_3 & s_+Q_{12} & 0\\
            s_+Q_{12} & -s_+Q_{11}-q_3 & 0\\
            0 & 0 & 2q_3
        \end{pmatrix},
\end{equation}
satisfies the boundary conditions \eqref{eq:3D_bcs}. Moreover, \eqref{eq:mapping} satisfies the corresponding fully 3D system of Euler-Lagrange equations at this special temperature. We plot the biaxiality parameter, eigenvalues and director (i.e., the eigenvector of $\Qvec_f$ with a non-degenerate eigenvalue if $\Qvec_f$ is uniaxial, and the eigenvector of $\Qvec_f$ with the largest positive eigenvalue if $\Qvec_f$ is biaxial), of \eqref{eq:mapping} in \Cref{fig:eigenvalues}. We have two regions of biaxiality (including two points of maximal biaxiality) encompassing a point of uniaxiality at the channel centre, which corresponds to a domain wall. Moreover, the director escapes into the third dimension with a negative corresponding eigenvalue at the channel centre making the transition across this domain wall continuous. 


\begin{figure}[ht]
    \centering
    \begin{minipage}{0.6\textwidth}
        \centering
        \includegraphics[width=1.0\textwidth]{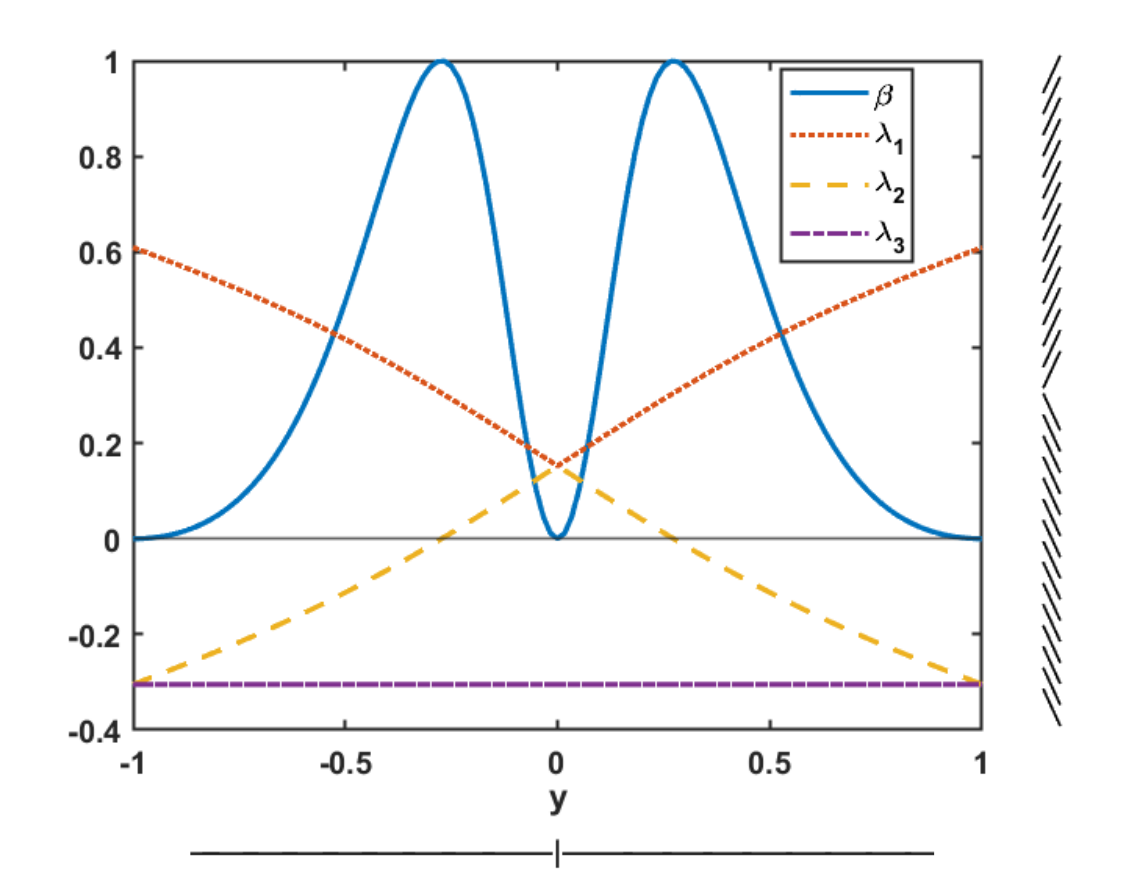}
    \end{minipage}
    \caption{The eigenvalues and biaxiality parameter $\beta$, of an OR solution of \eqref{eq:Q-EL-eqnts}, subject to the boundary conditions \eqref{eq:Q-bcs}, with $\omega=\frac{1}{4}$ and $L^*=\frac{1}{2}$. We map this OR solution of \eqref{eq:Q-EL-eqnts} to a full $\Qvec$-tensor via \eqref{eq:mapping} (above). On the right we plot the director in the $xy$-plane and below in the $yz$-plane.} 
    \label{fig:eigenvalues}
\end{figure}

\section{Supplementary material for section \ref{sec:constant-pressure-flow} - Passive flows with constant velocity and pressure}

Here we present supplementary material for section \ref{sec:constant-pressure-flow} of the main text. 

\begin{theorem}
(Maximum Principle) Let $s$ and $\theta$ be  
solutions of \eqref{eq:s-no-flow} and \eqref{eq:theta-no-flow}, where $s$ is at least $C^2$ and $\theta$ is at least $C^1$, then
\begin{equation}
     0< s\leq 1\quad\forall y\in[-1,1]. 
\end{equation}
\end{theorem}

\begin{proof} 
Let $(s,\theta)$ denote a solution pair of \eqref{eq:s-no-flow} and \eqref{eq:theta-no-flow}, and assume for contradiction that $s$ has a local minimum at $\hat{y}$, such that $s(\hat{y})\leq0$. This 
implies that $B=0$ using \eqref{eq:theta-no-flow}.
If $B=0$, then we must have $s=0$ or constant $\theta$, at every point in $\Omega$. 
This solution is determined by the ordinary differential equation: 
\begin{equation}
    \label{eq:sorder}
s^{\prime\prime}=\eps s(s^2-1),
\end{equation}
which can  be integrated to obtain the scalar order parameter. Doing this, we find
\begin{equation}
    s^\prime=\pm\sqrt{\left(\epsilon\left(\frac{s^4}{2}-s^2\right)+A\right)}.\label{eq:s-OR-eqtn}
\end{equation}
Evaluating at $s=1$, we see $A\geq\frac{\eps}{2}$. At the minimum, $s^\prime(\hat{y})=0$, hence
\begin{equation}
    s^2(\hat{y})=1\pm\sqrt{1-\frac{2A}{\epsilon}},
\end{equation}
which requires $A\leq\frac{\eps}{2}$. Combining these inequalities yields $A=\frac{\eps}{2}$. We then have
\begin{equation*}
    s^\prime
    =\pm\sqrt{\frac{\epsilon}{2}(s^2-1)^2}.
\end{equation*}
Fixing the sign in the above to be either positive or negative, we have a first order ODE subject to the boundary condition $s(-1)=1$, or $s(1)=1$. In any case, $s\equiv1$ is a solution, hence, by the Picard-Lindel\"of Theorem, this is the unique solution and this is clearly positive everywhere. 



We prove that $s\leq 1$ by a direct application of the maximum principle. Assume that there exists a point $y^*\in[-1,1]$ where $s$ attains its maximum, and $s(y^*)>1$ so that $y^*\in(-1,1)$. 
The function $s^2$ must also attain its maximum at the point $y^*\in(-1,1)$, 
so that
\begin{equation*}
    \left(s^2\right)^{\prime\prime}(y^*)\leq 0.
\end{equation*}
Next, note that $\left(s^2\right)^{\prime\prime}=2(s^\prime)^2+2ss^{\prime\prime}$. We now multiply \eqref{eq:s-no-flow} by $s$, and substitute for $s^{\prime\prime}s$ in the resulting expression to obtain
\begin{equation}
    \frac{1}{2}\left(s^2\right)^{\prime\prime}=(s^\prime)^2+4s^2(\theta^\prime)^2+\epsilon s^2(s^2-1). \label{eq:prop1}
\end{equation}
Using $s(y^*)>1$, \eqref{eq:prop1} implies that $\left(s^2\right)^{\prime\prime}(y^*)>0$, which is a contradiction. Hence, we conclude that
$
    s\leq 1\;\forall y\in[-1,1].
$
\end{proof}

\subsection{The \texorpdfstring{$\eps\to0$}{Lg} and \texorpdfstring{$\eps\to\infty$}{Lg} limits}

In the $\epsilon \to 0 $ limit, relevant for nano-scale channels (also see \cite{han2020siap})
, the limiting problem can be solved explicitly. 
Recall the system \eqref{eq:Q-EL-eqnts}. 
From the maximum principle, $||\Qvec||_{L^\infty}$ is bounded independently of $\epsilon$, and the system \eqref{eq:Q-EL-eqnts} reduces to the Laplace equations in the $\epsilon\to 0$ limit \cite{fang2019}:
\begin{equation*}
    Q^{\prime\prime}_{11}=0, \quad Q^{\prime\prime}_{12}=0.
\end{equation*}
This limiting system, subject to the boundary conditions (\ref{eq:Q-bcs}), admits the unique solution
\begin{equation}
    Q_{11}(y)=\frac{1}{2}\cos(2\omega\pi),\;Q_{12}(y)=\frac{y}{2}\sin(2\omega\pi). \label{eq:Q-lim}
\end{equation}
Substituting \eqref{eq:Q-lim} into \eqref{eq:s-theta-relationship} (the relationship between the $\Qvec$-components and $s, \theta$ - refer to the main manuscript), we obtain the following limiting profiles for $s$ and $\theta$, in the $\epsilon \to 0$ limit:
\begin{subequations}\label{eq:s-theta-large-domains}
\begin{align}
    & s_{0, \omega}=\sqrt{\cos^2(2\omega\pi)+y^2\sin^2(2\omega\pi)}, \label{eq:s-Q-solution}\\
    & \theta_{0, \omega}=\frac{1}{2}
    \textrm{atan2}(y\sin(2\omega\pi),\cos(2\omega\pi)).
    \label{eq:theta-Q-solution}
\end{align}
\end{subequations}
Using the explicit expressions above, one can easily verify that $s_{0,\omega}$ has exactly one global minimum at $y=0$. Further, $s_{0, \pm \frac{1}{4}}\left(0\right)=0$ and $s_{0, \omega}(0)>0$ for $\omega\neq\pm\frac{1}{4}$.

\begin{proposition}
For $\omega\neq\pm\frac{1}{4}$, $s_{0,\omega}$ has exactly one critical point at $y=0$, which is a non-trivial global minimum i.e. $s_{0, \omega}(0)>0$. For $\omega=\pm\frac{1}{4}$, $s$ has exactly one minimum at $y=0$, such that $s_{0, \pm \frac{1}{4}}\left(0\right)=0$. 
\end{proposition}

\begin{proof} It is clear from \eqref{eq:s-Q-solution} that $s_{0, \omega}(-y)=s_{0, \omega}(y)$ and as such $s_{0, \omega}$ is symmetric.  
We quickly note from \eqref{eq:s-Q-solution}, that $s=1$ when $\omega=0,\pm \frac{1}{2}$. 
Next, we consider the cases $\omega\neq 0,\;\pm\frac{1}{4}$, $\pm\frac{1}{2}$. 
Differentiating \eqref{eq:s-Q-solution}, we have
\begin{align*}
    s_{0, \omega}^\prime(y)&=\pm\frac{y\sin^2(2\omega\pi)}{\sqrt{\cos^2(2\omega\pi)+y^2\sin^2(2\omega\pi)}}=0
    \implies y=0\textrm{ since $\omega\neq0,\pm\frac{1}{2}$.}
\end{align*}
Hence, the solution has one critical point at $y=0$, which is a global minimum. Since $s(0) = \cos(2\omega\pi)$, this minimum is non-trivial for $\omega\neq\pm \frac{1}{4}$.

Next, we briefly consider the case when $\omega=\pm\frac{1}{4}$. From \eqref{eq:s-Q-solution} we see that the solution is given by
\begin{equation*}
s_{0, \pm \frac{1}{4}}(y)=
\begin{cases}
    -y\quad\textrm{for $y\in\left[-1,0\right]$}\\
    y\quad\textrm{for $y\in\left[0,1\right]$},
    \end{cases}
\end{equation*}
which clearly has a unique minimum value $y=0$, and $s_{0, \pm \frac{1}{4}}(0) = 0.$ We have a domain wall at $s=0$, and $\theta_{0,  \frac{1}{4}}= -\frac{\pi}{4}$ for $y<0$ and $\theta_{0, \frac{1}{4}}=\frac{\pi}{4}$ for $y>0$. Analogous remarks apply to $\omega = -\frac{1}{4}$. In other words, there are polydomain structures with distinct nematic directors, separated by a domain wall i.e. the unique limiting profile is an OR solution, and hence globally stable in the $\epsilon \to 0$ limit, for $\omega=\pm \frac{1}{4}$.
\end{proof}

In the $\epsilon \to \infty$ limit (relevant for micron-scale channels), the system \eqref{eq:EL-eqtns} (refer to main manuscript) reduces to (see \cite{braides} for rigorous arguments)
\begin{equation}
    s(s^2-1)=0,\quad s^2\theta_y=B,
\end{equation}
which, subject to the boundary conditions (\ref{eq:theta-bc}), has the solution
\begin{equation}
    s(y)=1, \quad \theta(y)=\omega\pi y \quad \textrm{for all $\omega$, including $\omega = \pm \frac{1}{4}$.}
\end{equation}
OR solutions are unstable in the $\epsilon \to\infty$ limit, for $\omega = \pm \frac{1}{4}$ \cite{lamy}. However, we can deduce the asymptotic profiles of OR solutions in this limit, since OR solutions exist for all $\epsilon$, when $\omega = \pm \frac{1}{4}$. To this end, we introduce the OR energy for $\omega = \pm \frac{1}{4}$ and $Q_{11}(y) \equiv 0 $ for all $y\in [-1, 1]$:
\begin{equation}
    E(Q_{12}):=\int_{-1}^1 \left(Q_{12}^\prime\right)^2+\eps Q_{12}^2(2Q^2_{12}-1),~\mathrm{d}y,\label{eq:OR_energy}
\end{equation}
subject to the boundary conditions in (\ref{eq:Q-bcs}).
As $\epsilon\to \infty$, the minimizers of the OR energy converge to the set $\mathcal{B}^{OR}$ where
\[
    \mathcal{B}^{OR} = \left\{ (Q_{11}, Q_{12}) = \left(0,\pm\frac{1}{2} \right)
    \right\}.
\]

Focusing on $\omega = \frac{1}{4}$ and replicating arguments from \cite{dalby-farrell-majumdar-xia}, we deduce that OR solutions, interpreted as minimizers of \eqref{eq:OR_energy}, converge in $ L^1\left([-1,1]\right)$, almost everywhere to a map of the form
\begin{equation}
    \Qvec^* = \left(0,-\frac{1}{2}\right)\chi_{E_1}+\left(0,\frac{1}{2}\right)\chi_{E_2},\label{eq:Q_0} 
\end{equation}
where $\chi$ is the characteristic function of an interval, $E_1=[-1,0)$ and $E_2=(0,1]$, in the $\epsilon \to \infty$ limit.

\section{Supplementary material for section \ref{sec:numerics} - Numerical methods}
Here we explain the numerical methods used in section \ref{sec:numerics} of the main text.
We write the dynamical systems in their weak formulation. For example, the weak formulation of the complicated active system \eqref{eq:Q-active}, is the following: 
\begin{subequations}
\begin{align}
        & \int_{-1}^1\frac{\partial Q_{11}}{\partial t}v_1dy=\int_{-1}^1 u_y Q_{12}v_1- Q_{11,y}v_{1y}+\frac{1}{L^*}Q_{11}(1-4(Q_{11}^2+Q_{12}^2))v_1~\mathrm{d}y,
        \label{a}\\
        & \int_{-1}^1\frac{\partial Q_{12}}{\partial t}v_2dy=\int_{-1}^1 -u_y Q_{11}v_2- Q_{12,y}v_{2y}+\frac{1}{L^*}Q_{12}(1-4(Q_{11}^2+Q_{12}^2))v_2~\mathrm{d}y,
        \\
        & \int_{-1}^1\frac{\partial u}{\partial t}v_3dy=\int_{-1}^1
        -p_xv_3-\left(u_{y}+2L_2(Q_{11}Q_{12,yy}-Q_{12}Q_{11,yy})+\Gamma(Q_{12}c^2)\right)v_{3y}~\mathrm{d}y,\label{c}
\end{align}
\end{subequations}
for all $v_1,v_2,v_3\in W^{1,2}_0([-1,1])$ with Dirichlet boundary conditions for $(Q_{11},Q_{12})$ and $u$, given in \eqref{eq:Q-bcs} and \eqref{eq:poiseuille}, respectively. 
We partition the domain $[-1,1]$ into a uniform mesh with mesh size $h = 1/256$. We solve the dynamical systems \eqref{a}-\eqref{c} with finite element methods and Newton solver. Due to the third order partial derivatives with respect to $y$ in (\ref{eq:Q-active}), Lagrange elements of order $2$ are used for the spatial discretization. 

We also study the linear stability of the equilibrium solutions in \eqref{eq:Q-flow} and \eqref{eq:Q-active}. The systems can be written as $\frac{\partial \mathbf{x}}{\partial t} = F(\mathbf{x}(t))$.
Let $\mathbf{x}_0$ denote an equilibrium point i.e. $F( \mathbf{x}_0) = \mathbf{0}$, and let $J(\mathbf{x}_0) = \nabla F(\mathbf{x}_0)$ be the Jacobian matrix of $F$ at $\mathbf{x}_0$. We can then determine the stability of $\mathbf{x}_0$ by checking the sign of the largest real part amongst all eigenvalues of $J(\mathbf{x}_0)$. If the largest real part is negative (positive), then the equilibrium point is stable (unstable).

For stable states of the system \eqref{eq:Q-flow}, we use the semi-implicit Euler method for time discretization and the initial conditions
\begin{gather}
Q_{11} = \cos(2\omega \pi y)/2,\ 
Q_{12} = \sin(2\omega \pi y)/2,\ 
u = -p_x(1-y^2)/2.
\end{gather}
For the unstable OR-type solutions, we assume that the partial derivatives with respect to $t$ are zero, and solve the passive or active flow systems using a Newton solver with a linear LU solver at each iteration. Newton's method strongly depends on the initial condition, so we use the asymptotic expressions \eqref{eq:s-comp} and \eqref{eq:theta-asymptotic} as initial conditions for the passive flow system, and \eqref{eq:s-comp} and \eqref{eq:theta-active-explicit} as initial conditions for the active flow system with small $\Gamma$. In the active case, we perform an increasing $\Gamma$ sweep for the OR branch to obtain OR-type solutions for large $\Gamma$. 

\section{Supplementary material for section \ref{sec:active} - Constant concentration assumption }
In section \ref{sec:active}, we consider the system of active equations \eqref{eq:active_system}, in the case of constant concentration $c$, and therefore omit \eqref{eq:c_dimensional_eqtn}. Here we give justification to this assumption. First, we note that in \cite{Giomo_annihilation}, the authors numerically solve the system \eqref{eq:active_system} (including \eqref{eq:c_dimensional_eqtn}) in two-dimensions, for $\alpha_1= 0.1$ and $\alpha_2=0.2$, and they find the numerical solution $c/c_0$ (where $c_0$ is the average concentration of the system) lies in the interval $[0.96,1]$, therefore it is approximately constant. Similarly, in \cite{excitable}, the authors note the variation in $c$ is of the order of 2\% for their numerically computed profiles when $\alpha_1=0.2$.

Next, we numerically solve the system \eqref{eq:active_system} in the case of non-constant concentration, in our one-dimensional framework. Using the scalings outlined in the main text, the corresponding dimensionless equations are
\begin{subequations}\label{eq:non_constant_c}
\begin{align}
        &\frac{\partial c}{\partial t}= c_{yy}(\Gamma_0-\Gamma_1 Q_{11})-\Gamma_1 c_y Q_{11,y}-\hat{\alpha}_1(2cc_y Q_{11,y}+c^2 Q_{11,yy}),\label{eq:1}\\
        & \frac{\partial Q_{11}}{\partial t}=u_y Q_{12}+ Q_{11,yy}+\left(\frac{c-c^*}{2}\right)Q_{11}-2cQ_{11}(Q_{11}^2+Q_{12}^2), \\
        & \frac{\partial Q_{12}}{\partial t}=-u_y Q_{11}+ Q_{12,yy}+\left(\frac{c-c^*}{2}\right)Q_{11}-2cQ_{11}(Q_{11}^2+Q_{12}^2), \\
        & L_1\frac{\partial u}{\partial t}=-p_x+u_{yy}+\frac{2\gamma}{\mu}(Q_{11}Q_{12,yy}-Q_{12}Q_{11,yy})_y + \hat{\alpha}_2(c^2 Q_{12,y}+2cc_y Q_{12})\label{eq:4},    
\end{align}
\end{subequations}
where $\Gamma_0=\frac{D_0 \gamma}{\kappa}$, $\Gamma_1=\frac{D_1 \gamma}{\kappa}$, $\hat{\alpha}_1=\frac{\alpha_1\gamma}{L^2 \kappa}$, $L_1=\frac{\rho \kappa}{\gamma\mu}$ and $\hat{\alpha}_2=\frac{\alpha_2 \gamma }{\kappa\mu L^2}$ are dimensionless parameters. Here, $\hat{\alpha}_1$ and $\hat{\alpha}_2$ are our measures of activity so that small/large values correspond to small/large activity.  In \Cref{fig:non_constant_c}, we present some numerical solutions of \eqref{eq:1}-\eqref{eq:4}, subject to the following boundary conditions
\begin{equation}
    s(\pm 1)=0.5,\; \theta(\pm 1)=\pm\frac{\pi}{4},\; c_y(\pm 1)=0,
\end{equation}
which can be translated into conditions for $\Qvec$ as in \eqref{eq:Q-bcs}. With $\hat{\alpha}_1=\hat{\alpha}_2=$1, 10, 50, 500, 1000 we clearly recover solutions with approximately constant concentration that are still OR-type profiles, supporting our claim that constant $c$ is a reasonable assumption in certain parameter regimes and/or for OR-type solutions.
\begin{figure}[ht]
    \centering
    \begin{minipage}{1.0\textwidth}
        \centering
        \includegraphics[width=1.0\textwidth]{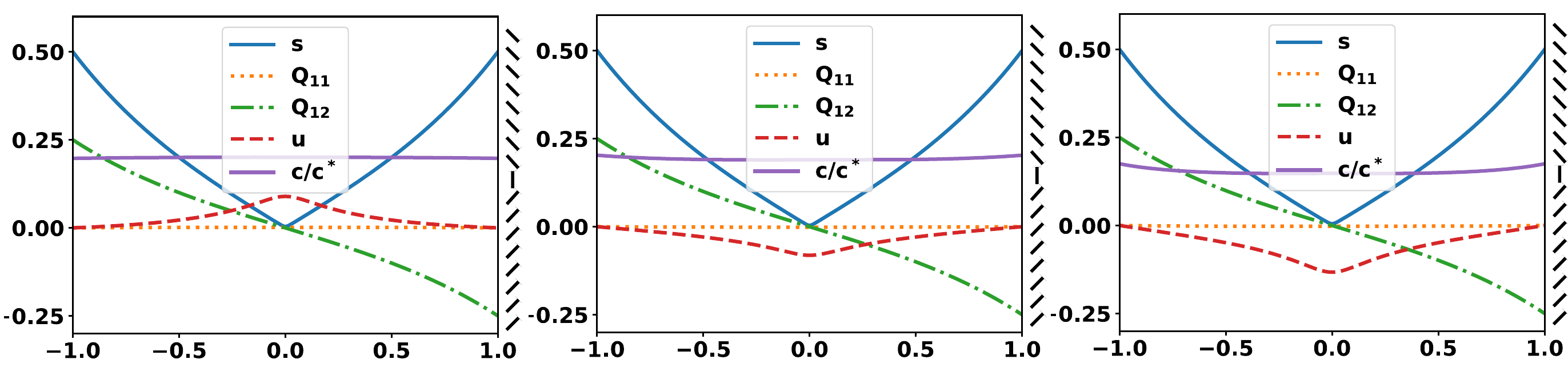}\\
    \end{minipage}
    \begin{minipage}{0.7\textwidth}
        \centering
        \includegraphics[width=0.9\textwidth]{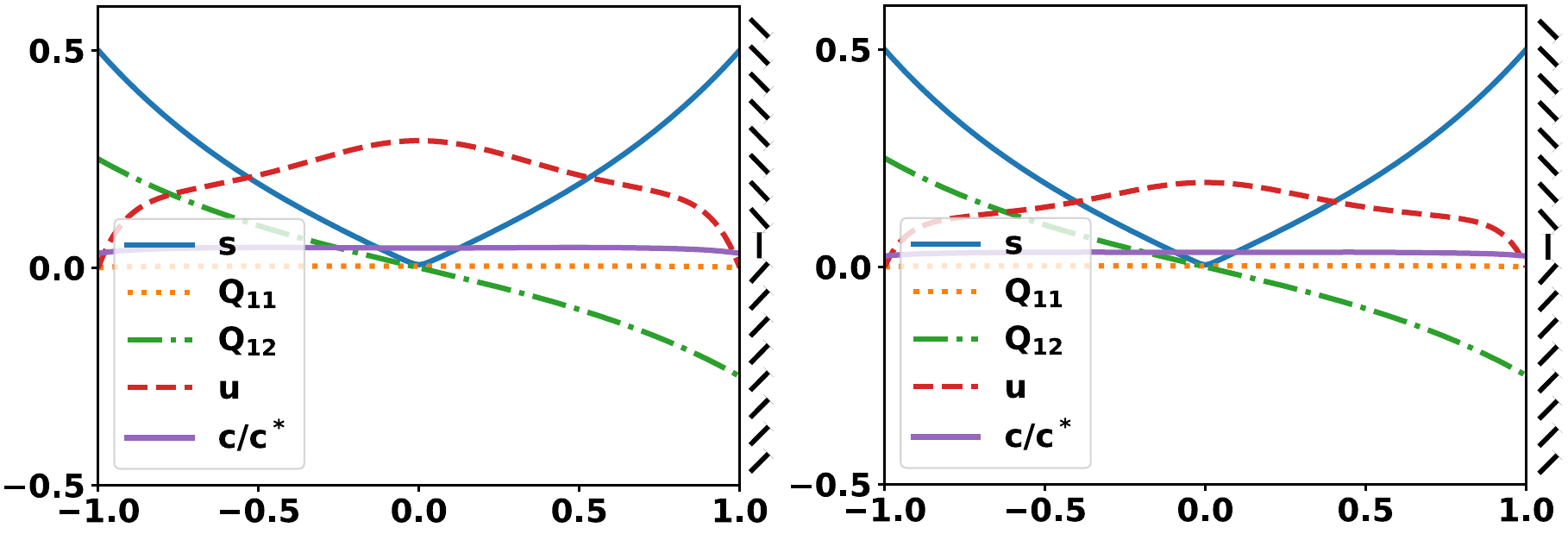}\\
    \end{minipage}
    
    \caption{Solutions of the system \eqref{eq:non_constant_c}, for $\Gamma_0=\Gamma_1=L_1=1$, $\gamma/\mu=0.1$, $p_x=-5$, $c^*=\sqrt{3\pi/2}$, $\hat{\alpha_1}=\hat{\alpha}_2=1,10,50,500,1000$ (from left to right) and $\omega=-1/4$.}
    \label{fig:non_constant_c}
    \end{figure}

\end{document}